\documentclass[11pt]{amsart}

\usepackage{amsfonts}
\usepackage{amsmath}
\usepackage{amsthm}
\usepackage[mathscr]{euscript}
\usepackage{mathrsfs}
\usepackage{indentfirst}
\usepackage{booktabs,multirow,bigstrut,makecell}
\usepackage{color}\usepackage{enumitem}

\newtheorem{theo}{Theorem}[section]
\newtheorem{lem}[theo]{Lemma}
\newtheorem{coro}[theo]{Corollary}

\newtheorem{rem}[theo]{Remark}
\newtheorem{prop}[theo]{Proposition}
\numberwithin{equation}{section}

\newcommand{\lbl}[1]{\label{#1}}
\allowdisplaybreaks

\newcommand{\be}{\begin{equation}}
\newcommand{\ee}{\end{equation}}
\newcommand\bes{\begin{eqnarray}} \newcommand\ees{\end{eqnarray}}
\newcommand{\bess}{\begin{eqnarray*}}
\newcommand{\eess}{\end{eqnarray*}}
\newcommand{\bbbb}{\left\{\begin{aligned}}
\newcommand{\nnnn}{\end{aligned}\right.}
\newcommand{\bea}{\begin{align*}}
\newcommand{\eea}{\end{align*}}

\newcommand\ep{\varepsilon}

\newcommand\dd{\displaystyle}

\newcommand\dx{{\rm d}x}
\newcommand\dz{{\rm d}z}

\newcommand\dy{{\rm d}y}
\newcommand\dt{{\rm d}t}

\newcommand\lm{\lambda}

\newcommand\bbb{\big}

\newcommand\R{\mathbb{R}}
\newcommand\ol{\overline}

\newcommand\sk{\smallskip}

\begin{document}
\def\theequation{\arabic{section}.\arabic{equation}}

\setlength{\baselineskip}{16pt} \pagestyle{myheadings}

\title[Nonlocal diffusion systems with free boundaries]{Two species nonlocal diffusion systems with free boundaries}

\author[Y.H. Du, M.X. Wang and M. Zhao]{Yihong Du$^\dag$, Mingxin Wang$^\ddag$ and Meng Zhao$^\S$}
\thanks{$^\dag$School of Science and Technology, University of New England, Armidale, NSW 2351, Australia. {\sf E-mail: ydu@une.edu.au.} This author was supported by the Australian Research Council}
\thanks{$^\ddag$({\sf Corresponding author}) School of Mathematics, Harbin Institute of Technology, Harbin 150001, PR China. {\sf E-mail: mxwang@hit.edu.cn.} This author was supported by NSFC Grant 11371113}
\thanks{$^\S$School of Mathematics and Statistics, Lanzhou University, Lanzhou, Gansu 730000, PR China. This author was supported by a scholarship from the China Scholarship Council.}

\date{\today}

\begin{abstract} We study a class of free boundary systems with nonlocal diffusion, which are natural extensions of the corresponding free boundary problems of reaction diffusion systems. As before the free boundary represents the spreading front of the species, but here the population dispersal is described by ``nonlocal diffusion'' instead of ``local diffusion". We prove that such a nonlocal diffusion problem with free boundary has a unique global solution, and for models with Lotka-Volterra type competition or predator-prey growth terms,
we show that a spreading-vanishing dichotomy holds, and obtain criteria for spreading and vanishing; moreover, for the weak competition case and for the weak predation case, we can determine the long-time asymptotic limit of the solution when spreading happens.  Compared with the single species free boundary model with nonlocal diffusion considered recently in \cite{CDLL},
and the two species cases with local diffusion extensively studied in the literature, the situation considered in this paper
involves several extra difficulties, which are overcome by the use of some new techniques.


\end{abstract}
\maketitle

\section{Introduction}
\setcounter{equation}{0} {\setlength\arraycolsep{2pt}

Nonlocal diffusion has been widely used to describe diffusion processes where  long range dispersal may play a significant role, a situation arising frequently in propagation questions  in biology and ecology (see, e.g., \cite{Nathan12}). Several well-known population models, where population dispersal was traditionally approximated by local diffusion, have been examined recently with the local diffusion operator in the model replaced by a nonlocal diffusion operator;
 see, for example, \cite{BL, BLS, BZ07, B-JMB16, HNS, H-JMB03, KLS10} and references therein. A commonly used nonlocal diffusion operator
 has the form
 \[d(J * u-u)(t,x):=d\left(\int_{\mathbb R^N} J(x-y)u(t,y)dy-u(t,x)\right),\]
where the kernel function $J: \mathbb R\to \mathbb R$ is continuous, nonnegative, even, and $\int_{\mathbb R} J(x)dx=1$. The quantity $J(x-y)$ is proportional to the probability that an individual member of the species (whose population density is $u(t,x)$) in location $x$ moves to location $y$ or vice versa.

In \cite{CDLL}, such a nonlocal diffusion operator was applied to the free boundary model of \cite{DLin10}, to investigate the spreading behaviour of a new or invasive species. The nonlocal diffusion model with free boundary in \cite{CDLL} has the form\vspace{-3mm}
 \begin{equation}\left\{\begin{aligned}
&u_t=d\int_{g(t)}^{h(t)}J(x-y)u(t,y)\dy-du(t,x)+f(t,x,u), & &t>0,~g(t)<x<h(t),\\
&u(t,g(t))=u(t,h(t))=0,& &t>0,\\
&h'(t)=\mu\int_{g(t)}^{h(t)}\int_{h(t)}^{\infty}
J(x-y)u(t,x)\dy\dx,& &t>0,\\
&g'(t)=-\mu\int_{g(t)}^{h(t)}\int_{-\infty}^{g(t)}
J(x-y)u(t,x)\dy\dx,& &t>0,\\
&u(0,x)=u_0(x),& &|x|\le h_0,\\
&h(0)=-g(0)=h_0,
 \end{aligned}\right.
 \label{1.1}\end{equation}
where $x=g(t)$ and $x=h(t)$ are the moving boundaries
to be determined together with $u(t,x)$, which is always assumed to be identically 0 for $x\in \R\setminus [g(t), h(t)]$; $d$, $\mu$ and $h_0$ are given positive constants. The kernel function $J: \mathbb{R}\to\mathbb{R}$  satisfies
 \begin{enumerate}[leftmargin=4em]
\item[{\bf(J)}] \;\; $J$  is continuous, nonnegative and even,\;\;  $J(0)>0,~\dd\int_{\mathbb{R}}J(x)\dx=1$, \ \ $\dd\sup_{\R}J<\infty$.
 \end{enumerate}
The growth function $f(t,x,u)$ is continuous, locally Lipschiz in $u$, and $f(t,x, 0)\equiv 0$.

In \cite{CDLL}, the existence and uniqueness of a global solution were proved, and for the special case that $f=f(u)$ is a logistic function, a spreading-vanishing dichotomy, criteria for spreading and vanishing, and long time behaviour of the solution were established.  A series of new ideas and techniques appeared in \cite{CDLL}.

In this paper we further develop the ideas and techniques in \cite{CDLL} to study systems of population models with nonlocal diffusion and free boundaries. It turns out that extra difficulties arise, and further new techniques are required. In order to keep the presentation transparent and ideas clear,
we will restrict to systems with only two species.

We consider the case that the two  species under consideration spread through a common spreading front, as in \cite{GW, Wjde14, WZjdde14}. Such a setting arises rather naturally in several situations; for example, when the two species are of predator-prey type, with the predator following (or driving) the spreading of the prey, or for two competing plant species
whose spreading relies on the same group of animals (insects, birds etc.) carrying their seeds to new fields.
Based on the  free boundary conditions in \eqref{1.1} above, this free boundary problem with nonlocal diffusion can be expressed in  the form
 \bes
\left\{\begin{aligned}
&u_{it}=d_i\dd\int_{g(t)}^{h(t)}J_i(x-y)u_i(t,y)\dy-d_iu_i+f_i(t, x, u_1, u_2), &&t>0,~g(t)<x<h(t),\\
&u_i(t,g(t))=u_i(t,h(t))=0, &&t\ge 0,\\
&h'(t)=\dd\sum_{i=1}^2\mu_i \int_{g(t)}^{h(t)}\!\int_{h(t)}^{\infty}
J_i(x-y)u_i(t,x)\dy\dx, &&t\ge 0,\\
&g'(t)=-\dd\sum_{i=1}^2\mu_i\int_{g(t)}^{h(t)}\!\int_{-\infty}^{g(t)}\!
J_i(x-y)u_i(t,x)\dy\dx,\ &&t\ge 0,\\
&u_i(0,x)=u_{i0}(x),\ \ \ h(0)=-g(0)=h_0, &&|x|\le h_0,\\
&i=1,\,2,
\end{aligned}\right.
 \label{1.2}
 \ees
where $x=g(t)$ and $x=h(t)$ are the moving boundaries to be determined together with $u_1(t,x)$ and $u_2(t,x)$, which are always assumed to be identically $0$ for $x\in \R\setminus [g(t), h(t)]$; $d_i$ and $\mu_i$ ($i=1,2$) are positive constants.
We assume that the initial function pair $(u_{10}, u_{20})$ satisfies
\begin{equation}
u_{i0}\in C([-h_0,h_0]),~ \ u_{i0}(\pm h_0)=0, ~ ~u_{i0}>0~\ \text{ in }~(-h_0,h_0), \ \ i=1,\,2,
\label{1.3}
\end{equation}
with  $[-h_0,h_0]$ representing the initial population range of the species. The kernel functions $J_1$ and $J_2$   satisfy the condition {\bf(J)}.

The free boundary conditions in \eqref{1.2} mean that the expansion rate of the common population range of the two species is proportional to the outward flux of the population of the two species; some  justifications of this assumption can be found in \cite{CDLL}.

The growth terms $f_i$ ($i=1,2$) are assumed to be continuous and satisfy
 \begin{enumerate}[leftmargin=4em]
\item[{\bf(f)}] $f_1(t,x,0,u_2)=f_2(t,x,u_1,0)=0$, and $f_i(t,x,u_1, u_2)$
is locally Lipschitz in $u_1,u_2\in\R^+$, i.e., for any $K_1, K_2>0$, there exists a constant $L(K_1, K_2)>0$ such that
 \[|f_i(t,x,u_1,u_2)-f_i(t,x,v_1,v_2)|\le L(K_1,K_2)(|u_1-v_1|+|u_2-v_2|)\]
for all $u_1, v_1\in [0, K_1]$, $u_2, v_2\in [0, K_2]$ and all $(t,x)\in \R^+\times \R$. When $K_1=K_2$, we write $L(K_1, K_2)=L(K_1)$;
\item[{\bf(f1)}] There exist $k>0$ and $r>0$ such that for all $u_2\ge 0$ and $(t,x)\in \R^+\times \R$, there hold: $f_1(t,x,u_1,u_2)<0$ when $u_1>k$, $f_1(t,x,u_1,u_2)\le ru_1$ when $0<u_1\le k$;
\item[{\bf(f2)}] For any given $K>0$, there exists $\Theta(K)>0$ such that $f_2(t,x,u_1,u_2)<0$ for $0\le u_1\le K$, $u_2\ge \Theta(K)$ and $(t,x)\in\R^+\times\R$.
 \end{enumerate}

\smallskip

We note that condition {\bf(f)} implies
 \[|f_1(t,x,u_1,u_2)|\le L(K_1,K_2)|u_1|, \ \ |f_2(t,x,u_1,u_2)|\le L(K_1,K_2)|u_2|\]
 for all $u_1\in[0, K_1]$, $u_2\in[0, K_2]$ and all $(t,x)\in \R^+\times \R$.

It is easily seen that the conditions {\bf(f)}, {\bf(f1)} and {\bf(f2)} hold for the following classical Lotka-Volterra competition and predator-prey growth terms:
 \bes
  &&\hspace{-0.8cm}\mbox{\bf Competition\, Model}: \ f_1=u_1(a_1-b_1u_1-c_1u_2), \ \ f_2=u_2(a_2-b_2u_2-c_2u_1),
  \lbl{1.4}\\[1mm]
 &&\hspace{-0.8cm}\mbox{\bf Predator-prey\, Model}: \ f_1=u_1(a_1-b_1u_1-c_1u_2), \ \ f_2=u_2(a_2-b_2u_2+c_2u_1),
  \lbl{1.5}\ees
where $a_i, b_i, c_i$ ($i=1,2$) are positive constants.

Unless otherwise stated, we always assume that $f_1$ and $f_2$ satisfy {\bf (f)}, {\bf (f1)} and {\bf (f2)}, $J_1,\,J_2$ satisfy  {\rm \bf (J)}, and
\eqref{1.3} is satisfied by the initial function pair. We will write
 \[\mbox{$\|a,b\|\le M$ to mean $\|a\|\le M$, $\|b\|\le M$.}\]

The main results of this paper are the following theorems.
 \begin{theo}\label{th1.1}\,Problem \eqref{1.2} has a unique  solution $(u_1, u_2, g, h)$ defined for all $t>0$.
 \end{theo}

\begin{theo}[Spreading-vanishing dichotomy]\label{th1.2}\, Assume further that $J_1(x)>0$, $J_2(x)>0$ in $\mathbb{R}$, and  that $(f_1,\,f_2)$ satisfies either \eqref{1.4} or \eqref{1.5}. Let $(u_1, u_2, g, h)$ be the unique solution of \eqref{1.2}. Then one of the following alternatives must happen:
\begin{itemize}
\item[{\rm(i)}] {\rm Spreading:} $\dd\lim_{t\to\infty} [h(t)-g(t)]=\infty$,
\item[{\rm(ii)}] {\rm Vanishing:} $\dd\lim_{t\to\infty} (g(t), h(t))=(g_\infty, h_\infty)$ is a finite interval and $\dd\lim_{t\to\infty}\max_{g(t)\le x\le h(t)}u_i(t,x)=0$, $i=1,\, 2$.
\end{itemize}
\end{theo}

\begin{theo}[Spreading-vanishing criteria]\label{th1.3}
Under the conditions of Theorem \ref{th1.2}, the following conclusions hold:
\begin{itemize}
\item[\rm (i)] If either $a_1\ge d_1$ or $a_2\ge d_2$, then spreading always happens.
\item[\rm (ii)] If $a_1<d_1$ and $a_2<d_2$, then there exists a unique $\ell_*>0$ such that
\begin{enumerate}
\item[\rm (a)] whenever vanishing happens, we have $h_\infty-g_\infty\le\ell_*$,
\item[\rm (b)] spreading always happens when $h_0\geq \ell_*/2$,
\item[\rm (c)] if  $h_0<\ell_*/2$,  then there exist two positive numbers $\Lambda^*\ge\Lambda_*>0$ such that vanishing happens when $\mu_1+\mu_2\leq\Lambda_*$ and spreading happens when $\mu_1 +\mu_2>\Lambda^*$.
\end{enumerate}
\end{itemize}
\end{theo}

As we will see in Section 3 below,  $\ell_*$ depends only on $a_i,\,d_i$ and  $J_i$, $i=1, 2$. On the other hand, $\Lambda_*$ and $\Lambda^*$ depend also on $b_i$, $c_i$ and $u_{i0}$, $i=1,2$.

To determine the long-time behaviour of the solution when spreading happens, we restrict to two special cases:
 \begin{itemize}
\item[(a)] \underline{The weak competition case}: $(f_1, f_2)$ satisfies \eqref{1.4} with
${b_1}/{c_2}>{a_1}/{a_2}>{c_1}/{b_2}$.\vskip 4pt
\item[(b)] \underline{The weak predation case}: $(f_1, f_2)$ satisfies \eqref{1.5} with
$a_1b_1b_2>a_2b_1c_1+a_1c_1c_2$.
 \end{itemize}

\begin{theo}[Asymptotic limit]\label{th1.4}
Let $(u_1,u_2,g,h)$ be the unique solution of \eqref{1.2} and suppose $\dd\lim_{t\to\infty}[h(t)-g(t)]=\infty$. Then
\begin{itemize} \item[\rm (i)] in the weak competition case we have
 \[\lim_{t\to\infty}(u_1(t,x), u_2(t,x))=\left(\frac{a_1b_2-a_2c_1}{b_1b_2-c_1c_2}, \ \frac{a_2b_1-a_1c_2}{b_1b_2-c_1c_2}\right) \ \mbox{ locally uniformly for } x\in\mathbb R,\]
 \item[\rm (ii)] in the weak predation case we have
  \[\lim_{t\to\infty}(u_1(t,x), u_2(t,x))=\left(\frac{a_1b_2-a_2c_1}{b_1b_2+c_1c_2}, \ \frac{a_1c_2+a_2b_1}{b_1b_2+c_1c_2}\right) \ \mbox{ locally uniformly for } x\in\mathbb R.\]
  \end{itemize}
   \end{theo}

\begin{rem}
We believe that the  condition $J_i(x)>0$ in $\mathbb R$ for $i=1,2$ in Theorem \ref{th1.2} is unnecessary, though our proof of $\dd\lim_{t\to\infty}\max_{g(t)\le x\le h(t)}u_i(t,x)=0$ in part (ii) of Theorem \ref{th1.2} makes essential use of this extra condition.
\end{rem}

\begin{rem}
When spreading happens, it is a challenging task to determine the long-time limit of the solution for systems with Lotka-Volterra growth terms in general. Many technical difficulties arising here do not occur in the corresponding free boundary problems with local diffusion. It appears that new techniques are needed to handle most of the cases not covered in Theorem {\rm\ref{th1.4}}.
\end{rem}

\begin{rem} When the nonlocal diffusion term in \eqref{1.2} is replaced by the usual local diffusion term $d_i \partial_{xx} u_i$, for competition and predator-prey type Lotka-Volterra growth functions $(f_1, f_2)$, the problem was investigated in \cite{GW, Wjde14, WZjdde14, ZhaoW16}. The results in Theorem {\rm\ref{th1.4}} indicate that when local diffusion is replaced by nonlocal diffusion for these special Lotka-Volterra systems with free boundary, the basic features of the model is not altered significantly. However, Theorem {\rm\ref{th1.3}\,(i)} suggests that the dispersal rates $d_1$ and $d_2$ play a more dominant role in determining whether the species can spread successfully
than in the local diffusion case  \cite{GW, Wjde14, WZjdde14, ZhaoW16}, reinforcing the phenomenon revealed in the single species case in \cite{CDLL}.
\end{rem}

The rest of the paper is arranged as follows. In Section 2 we prove Theorem \ref{th1.1}, namely  problem \eqref{1.2} has a unique global solution, by further developing the approach of \cite{CDLL}. As the situation here is more complicated, considerable changes are needed. In Section 3, we prove Theorems 1.2, 1.3 and \ref{th1.4}. Here we encounter a difficulty in understanding the vanishing case, which does not occur in the corresponding local diffusion systems with free boundary (see \cite{GW, Wjde14, WZjdde14, ZhaoW16}), or in the nonlocal diffusion model with a single species considered in \cite{CDLL}. To overcome this difficulty, we introduce a new technique; see details in the proof of Theorem \ref{th3.3}. The proof of the other conclusions is largely based on adequate adaptations of  techniques developed for the local diffusion case in \cite{Wjde14, WZjdde14, WZjdde17, ZW2015}.

\section{Global existence and uniqueness}
\setcounter{equation}{0} {\setlength\arraycolsep{2pt}

In this section we prove that, for any given initial value  $U_0:=(u_{10}, u_{20})$ satisfying \eqref{1.3}, problem (\ref{1.2}) has a unique solution defined for all $t>0$.
For convenience, we first introduce some notations. For given $h_0, T>0$, define
 \bess
\mathbb H_{h_0}^T&=&\left\{h\in C^1([0,T])~:~h(0)=h_0,
\; h(t) \mbox{ is strictly increasing}\right\},\\[1mm]
\mathbb G_{h_0}^T&=&\left\{g\in C^1([0,T])~:-g\in\mathbb{H}_{h_0}^T\right\}.
 \eess
For $g\in \mathbb G_{h_0}^T$, $h\in\mathbb H_{h_0}^T$ and $U_0=(u_{10}, u_{20})$ satisfying \eqref{1.3}, we denote
 \bess
 &D_T=D^T_{g, h}:=\left\{(t,x)\in\mathbb{R}^2:\, 0<t\leq T,~g(t)<x<h(t)\right\},&\\[2mm]
 &\mathbb X^T_{U_{0}}=\mathbb X^T_{U_{0},g,h}:=\left\{\varphi\in [C(\overline D_T)]^2:~\varphi\ge 0,\;  \varphi\big|_{t=0}=U_{0}(x),\; \varphi\big|_{x=g(t), h(t)}=0\right\}.&
 \eess
 Here by $\varphi=(\varphi_1,\varphi_2)\geq 0$ we mean $\varphi_1\geq 0$ and $\varphi_2\geq 0$ in $D_T$.

The following theorem, which contains the conclusion in Theorem \ref{th1.1}, is the main result of this section.

\begin{theo}\label{th2.1}\, For any given initial value $U_0:=(u_{10}, u_{20})$ satisfying \eqref{1.3},  problem \eqref{1.2} has a unique global solution $(u_1(t,x), u_2(t,x), g(t), h(t))$. Moreover, for any $T>0$, we have $(g, h)\in\mathbb G_{h_0}^T\times\mathbb H_{h_0}^T$, $(u_1, u_2)\in\mathbb{X}^T_{U_{0}, g, h}$, and
 \bes&&\left\{\begin{aligned}
 &0<u_1\le \max\left\{\|u_{10}\|_\infty,~k\right\}:=A_1 \ &&\mbox{in}\; \ D^T_{g, h},\\
 &0<u_2\le \max\left\{\|u_{20}\|_\infty,~\Theta(A_1)\right\}=:A_2 &&\mbox{in}\; \ D^T_{g, h},
 \end{aligned}\right.\label{2.1}\\[2mm]
 && g'(t)<0,\ \ \ h'(t)>0, \ \ \ \forall \ t>0.
 \lbl{2.2}\ees
\end{theo}

The rest of this section is devoted to the proof of Theorem \ref{th2.1}.
The following {\it Maximum Principle} will be used frequently in our analysis to follow.

\begin{lem}[Maximum Principle \cite{CDLL}]\label{l2.2}
Assume that $J$ satisfies {\rm \bf (J)}, and $(g, h)\in\mathbb G_{h_0}^T\times\mathbb H_{h_0}^T$. Suppose that $\psi, \psi_t\in C(\overline{D}_T)$ and satisfies, for some $c\in L^\infty (D_T)$,
 \bess\left\{\begin{aligned}
&\psi_t(t,x)\ge d\int_{g(t)}^{h(t)}J(x-y)\psi(t,y)\dy-du +c(t,x)\psi, && t\in (0, T],\ g(t)<x<h(t),\\
& \psi(t, g(t)) \geq 0,\ \psi(t, h(t)) \geq 0, && t>0,\\
&\psi(0,x)\ge0,  && |x|\le h_0.
 \end{aligned}\right.\eess
Then $\psi\ge0$ on $\overline{D}^T_{g, h}$. Moreover, if $\psi(0,x)\not\equiv0$ in $[-h_0, h_0]$, then $\psi>0$ in $D^T_{g, h}$.
\end{lem}

The following result will play a crucial role in the proof of Theorem \ref{th2.1}.

\begin{lem}\label{l2.3}\, For any $T>0$ and $(g, h)\in\mathbb G_{h_0}^T\times\mathbb H_{h_0}^T$, the problem
 \bes\left\{\begin{array}{lll}
w_{it}=d_i\dd\int_{g(t)}^{h(t)}J_i(x-y)w_i(t,y)\dy&&\\[2mm]
 \hspace{1.2cm}-d_iw_i(t,x)+f_i(t,x,w_1, w_2), \qquad
&& 0<t\le T,~g(t)<x<h(t),\\[3mm]
w_i(t,g(t))=w_i(t,h(t))=0, &&0\le t\le T,\\[1mm]
w_i(0,x)=u_{i0}(x),\ \ &&|x|\le h_0,\\
i=1,\,2&&
 \end{array}\right.
 \label{2.3}\ees
has a unique solution $w_{g,h}=(w_{1,g,h}, w_{2,g,h})\in\mathbb{X}^T_{U_{0}, g, h}$, and $w_{g,h}$ satisfies \eqref{2.1}.
 \end{lem}

\begin{proof} The idea of the proof comes from \cite{CDLL}. We break the proof into three steps.
\smallskip

{\bf Step 1:} {\it A parametrized ODE problem.} Define
 \begin{align*}
 f^*_i(t,x,u_1,u_2)=&\;\left\{\begin{aligned} &f_i(t,x,u_1,u_2) \ \ &&\mbox{ if } \ u_1, u_2\ge 0,\\
 &0&&\mbox{ if } \ u_1, u_2\le 0, \end{aligned}\right. \\
 f^*_1(t,x,u_1,u_2)=&\;f_1(t,x,u_1,0)\ \ \mbox{ if } \ u_2\le 0,\ u_1>0,\\
 f^*_2(t,x,u_1,u_2)=&\; f_2(t,x,0,u_2)\ \ \mbox{ if } \ u_1\le 0,\ u_2>0.
 \end{align*}
For any  given $x\in[g(T),h(T)]$, set
 \begin{align*}
\tilde u_{10}(x)=&\;\left\{\begin{aligned}
&0,& &|x|>h_0,\\
&u_{10}(x),& &|x|\le h_0,
\end{aligned}\right.\;\;\;
~~~\tilde u_{20}(x)=\left\{\begin{aligned}
&0,& &|x|>h_0,\\
&u_{20}(x),& &|x|\le h_0,
\end{aligned}\right.\\[2mm]
t_x=&\;\left\{\begin{aligned}
&t_{x,g}& &\mbox{ if }\ x\in[g(T),-h_0), \ x=g(t_{x,g}),\\
&0& &\mbox{ if } \ |x|\le h_0,\\
&t_{x,h}& &\mbox{ if } \ x\in(h_0,h(t)], \ x=h(t_{x,h}).
\end{aligned}
\right.
\end{align*}
Clearly $t_x=T$ for $x=g(T)$ or $x=h(T)$, and $t_x\in [0,T)$ for $x\in (g(T), h(T))$.

For any given $0<s\le T$ and $\varphi=(\varphi_1, \varphi_2)\in\mathbb{X}^s_{U_{0}}$, we first consider the initial value problem of the following ordinary differential system with parameter $x\in(g(s),h(s))$:
 \bes\label{2.4}
\left\{\begin{array}{lll}
p_{it}=d_i\dd\int_{g(t)}^{h(t)}J_i(x-y)\varphi_i(t,y)\dy-d_ip_i+f^*_i(t, x, p), \ \
& t_x<t\le s,\\[2mm]
p_i(t_x,x)=\tilde u_{i0}(x), \ \ &g(s)<x<h(s),\\
i=1,\, 2.
\end{array}\right.
 \ees
Denote
 \begin{align*}
F_i(t, x, p)=d_i\dd\int_{g(t)}^{h(t)}J_i(x-y)\varphi_i(t,y)\dy-d_ip_i+f^*_i(t, x, p),\ \ i=1,2,\\
 K_\varphi=1+A_1+A_2+\|\varphi_1,\,\varphi_2\|_{C(\overline D_s)}, \ \ L_\varphi=\max\big\{d_1, d_2\big\}+L(K_\varphi),
 \end{align*}
where $A_1$ and $A_2$ are given by \eqref{2.1}. Then for any $p_i, q_i\in (-\infty, K_\varphi]$,
$i=1,2$, we have
 \begin{align*}
 |F_i(t, x, p_1, p_2)-F_i(t, x, q_1, q_2)|
 \le&\; |f^*_i(t, x, p_1, p_2)-f^*_i(t, x, q_1, q_2)|
 +d_i|p_i-q_i|\\
 \le&\; L_\varphi\left(|p_1-q_1|+|p_2-q_2|\right),\ \ i=1,2.
 \end{align*}
In other words, the function $F_i(t,x,p)$ is Lipschitz continuous in $p=(p_1,p_2)$ for
$p_1,p_2\in (-\infty, K_\varphi]$ with Lipschitz constant $L_\varphi$, uniformly for $t\in [0, T]$
and $x\in [g(s), h(s)]$, $i=1,2$. Additionally, $F_i(t,x,p)$ is continuous in
all its variables in this range, $i=1,2$. Based on the Fundamental Theorem of ODEs, for every fixed $x\in (g(s), h(s))$, the problem \eqref{2.4} has a unique
solution $p^\varphi=(p_1^\varphi, p_2^\varphi)$ defined in some interval
$[t_x,T_x)$.

We will prove that $t\to p^\varphi(t,x)(t,x)$ can be uniquely extended to $[t_x, s]$. Clearly, it suffices to show that if $p^\varphi(t,x)$ is defined for $t\in[t_x, t_0]$ with $t_0\in (t_x, s]$, then
 \begin{equation}\label{2.5}
0\leq p_1^\varphi(t,x), \ p_2^\varphi(t,x)<K_\varphi \ \ \mbox{ for }\ t\in (t_x, t_0].
 \end{equation}
We first show that $p_1^\varphi(t,x)<K_\varphi$ in $(t_x, t_0]$. If this inequality is not true, then, by $p_1^\varphi(t_x,x)=\tilde u_{10}(x)\leq \|\varphi_1\|_{C(\overline D_s)}<K_\varphi$, there exists $t'\in (t_x, t_0]$ such that
$p_1^\varphi(t,x)<K_\varphi$ in $(t_x, t')$ and $p_1^\varphi (t',x)=K_\varphi$. It follows that $(p_1^\varphi)_t(t',x)\geq 0$ and $f_1^*(t', x, p_1^\varphi(t',x), p_2^\varphi(t',x))<0$ as
$K_\varphi>k$. Hence from the  equation satisfied by $p_1^\varphi$ we can deduce
 \[d_1K_\varphi=d_1p_1^\varphi(t',x)\leq d_1\int_{g(t')}^{h(t')}J_1(x-y)\varphi_1(t',y)\dy\leq d_1\|\varphi_1\|_{C(\overline{D}_T)}\leq d_1(K_\varphi-1).\]
This is a contradiction. Similarly, $p_2^\varphi(t,x)<K_\varphi$ in $(t_x, t_0]$.

We now prove the first inequality in \eqref{2.5}. Since
 \[|f^*_1(t,x,p_1,p_2^\varphi)|=|f^*_1(t,x,p_1,p_2^\varphi)-f^*_1(t,x,0,p_2^\varphi)|\leq L(K_\varphi)|p_1|,\ \ \forall\ p_1\in(-\infty, K_\varphi],\]
it follows that
 \[(p_1^\varphi)_t\geq c(t,x)p_1^\varphi+d_1\int_{g(t)}^{h(t)}J_1(x-y)\varphi_1(t,y)\dy
 \geq c(t,x)p_1^\varphi ,\ \ \forall\ t\in [t_x, t_0],\]
where $c(t,x)=-L(K_\varphi)-d_1$ when $p_1^\varphi(t,x)\ge 0$, and $c(t,x)=L(K_\varphi)$ when $p_1^\varphi(t,x)\le 0$. Notice $p_1^\varphi(t_x,0)=\tilde u_{10}(x)\geq 0$, the above inequality immediately gives $p_1^\varphi(t,x)\geq 0$ in $[t_x, t_0]$. Similarly, $p_2^\varphi(t,x)\geq 0$ in $[t_x, t_0]$. We have thus proved \eqref{2.5}, and therefore the solution $p^\varphi(t,x)$ of \eqref{2.4} is uniquely defined for $t\in [t_x, s]$.

\smallskip
{\bf Step 2:} {\it A fixed point problem.} Recall $p^\varphi(0,x)=U_0(x)$ for $|x|\le h_0$,
and $p^\varphi(t,x)=0$ for  $x\in\{g(t), h(t)\}$ and $t\in [0, s]$. Moreover, by the continuous dependence of the ODE solution on parameters, $p^\varphi$ is continuous in $\overline D_s$, and so $p^\varphi\in\mathbb{X}^s_{U_0}$.  Define a mapping $\Gamma_s: \mathbb{X}^s_{U_0}\to \mathbb{X}^s_{U_0}$ by
 \[\Gamma_s \varphi=p^{\varphi}.\]
Clearly, if $\Gamma_s\varphi=\varphi$ then $\varphi$ solves \eqref{2.4}, and vice versa.

We will show that $\Gamma_s$ has a unique fixed point in $\mathbb{X}^s_{U_0}$ when $0<s\ll 1$. This conclusion will be proved by the contraction mapping theorem, i.e., it will be shown that for such $s$, $\Gamma_s$ is a contraction on a closed subset of $\mathbb{X}^s_{U_0}$, and any fixed point of $\Gamma_s$ in $\mathbb{X}^s_{U_0}$ lies in this closed subset.

Take $C=\max\big\{2\|u_{10}\|_\infty,\, 2\|u_{20}\|_\infty,\, A_1,\,A_2\big\}$ and define
  \[\mathbf{X}^s_C:=\big\{\varphi=(\varphi_1,\varphi_2)\in\mathbb{X}^s_{U_0}: \; \|\varphi_1,\varphi_2\|_{C(\overline D_s)}\leq C\big\}.\]
Clearly $\mathbf{X}^s_C$ is a closed subset of $\mathbb{X}^s_{U_0}$. We will find a $\delta>0$ small depending on $C$ such that for every $s\in(0, \delta]$, $\Gamma_s$ maps $\mathbf{X}^s_C$ into itself, and is a contraction.

Let $\varphi\in\mathbf{X}^s_C$ and denote $p^\varphi=\Gamma_s \varphi$. Then $p^\varphi$ solves \eqref{2.4}, and so \eqref{2.5} holds with $t_0$ replaced by $s$. Thus, $f_i^*(t, x, p^\varphi)=f_i(t, x, p^\varphi)$ for $i=1,2$. Now we prove that for $0<s\ll 1$,
 \bes
 p_1^\varphi(t,x), \, p_2^\varphi(t,x)\leq C ,\ \ \forall\ g(s)\le x\le h(s),\ t_x\le t\le s,
 \lbl{2.6}\ees
which is equivalent to $\|p_1^\varphi,\,p_2^\varphi\|_{C(\overline D_s)}\leq C$. Note that  $p_1^\varphi, p_2^\varphi\ge 0$ implies $f_1(t, x, p_1^\varphi, p_2^\varphi)\leq rp_1^\varphi$. It follows from the first equation of \eqref{2.4} that, for $ t\in [t_x, s]$ and $x\in (g(s), h(s))$,
 \[(p_1^\varphi)_t\leq d_1\int_{g(t)}^{h(t)}J_1(x-y)\varphi_1(t,y)\dy+rp_1^\varphi\leq d_1\|\varphi_1\|_{C(\overline D_s)}+rp_1^\varphi.\]
Multiplying this inequality by $e^{-rt}$ and then integrating from $t_x$ to $t$ we obtain
 \begin{align*}
  p_1^\varphi(t, x)\le e^{r(t-t_x)}p_1^\varphi(t_x,x)+ d_1\int_{t_x}^t e^{r(t-\tau)}{\rm d}\tau \|\varphi_1\|_{C(\overline D_s)}
 \leq \|u_{10}\|_\infty e^{rs}+d_1Cs e^{rs }.
 \end{align*}
Take $\delta_1>0$ such that $d_1\delta_1 e^{r\delta_1}\leq 1/4$ and $e^{r\delta_1}\leq 3/2$. Then, for $s\in (0,\delta_1]$, we have
 \[ p_1^\varphi(t,x)\leq (8 \|u_{10}\|_\infty+C)/4\leq C \ \ \mbox{ in } \ D_s.\]
This combined with the properties of $f_2$ allows us to derive
 \[f_2(t, x, p_1^\varphi, p_2^\varphi)\le L(C, \Theta(C))p_2^\varphi:=L^* p_2^\varphi,\ \ \forall\ x\in (g(s), h(s)), \ t\in [t_x, s].\]
Similar to the above, take $\delta_2>0$ satisfying $d_2\delta_2 e^{L^*\delta_2}\leq 1/4$ and $ e^{L^*\delta_2}\leq 3/2$, then
 \[p_2^\varphi(t,x)\leq (8\|u_{20}\|_\infty+C)/4\leq C \ \ \mbox{ in } \ D_s\]
for all $s\in (0,\delta_2]$. Set $\delta=\min\{\delta_1, \delta_2\}$. Then \eqref{2.6} holds for $s\in (0,\delta]$.

\sk Thus $p^\varphi=\Gamma_s\varphi\in \mathbf X^s_C$, as desired.
Next we show that by shrinking $\delta$ if necessary, $\Gamma_s$ is a contraction on
$\mathbf X^s_C$ for $s\in (0, \delta]$. Let $\varphi,\rho\in\mathbf X^s_C$, then $p_i=p_i^\varphi-p_i^\rho$ satisfy
 \bess\left\{\begin{aligned}
 &p_{1t}+ap_1=d_1\dd\int_{g(t)}^{h(t)}J_1(x-y)\left(\varphi_1-\rho_1\right)(t,y)\dy
+bp_2, \ \ &&t_x<t\le s,~g(s)<x<h(s),\\
 &p_i(t_x,x)=0,&& g(s)<x<h(s),
 \end{aligned}\right.\eess
where $a=d_1-\partial_{p_1}f_1(t, x, \tilde p_1, p_2^\varphi)$, $b=\partial_{p_2}f_1(t, x, p_1^\rho, \tilde p_2)$ for some $\tilde p_i$ between $p_i^{\varphi}$ and $p_i^\rho$, $i=1,2$. Then $\|a,\,b\|_\infty\le d_1+L(C):=L_1$.
It follows that, for $x\in(g(s),h(s))$ and $t_x<t\le s$,
  \begin{align*}
p_1(t,x)=&\;e^{-\int_{t_x}^t\! a(\tau,x){\rm d}\tau}
\!\int_{t_x}^t\!e^{\int_{t_x}^l\! a(\tau,x){\rm d}\tau}\!\left(\!d_1\!\int_{g(l)}^{h(l)}\!J_1(x-y)\left(
\varphi_1-\rho_1\right)(l,y)\dy+b(l,x)p_2(l,x)\!\right){\rm d}l.
  \end{align*}
Since $(g(t), h(t))\subset(g(s), h(s))$ when $t\le s$, we deduce that, for $x\in(g(t),h(t))$,
\begin{align*}
 |p_1(t,x)|&\;\le e^{L_1(t-t_x)}\int_{t_x}^te^{L_1(l-t_x)}{\rm d}l\left(d_1\|\varphi_1-\rho_1\|_{C(
\overline D_s)}+L_1\|p_2\|_{C(\overline D_s)}\right)\\
&\;\le (t-t_x)e^{2L_1(t-t_x)}\left(d_1\|\varphi_1-\rho_1\|_{C(
\overline D_s)}+L_1\|p_2\|_{C(\overline D_s)}\right)\\
&\;\le s e^{2L_1s}\left(d_1\|\varphi_1-\rho_1\|_{C(
\overline D_s)}+L_1\|p_2\|_{C(\overline D_s)}\right).
 \end{align*}
This gives
 \[\|p_1\|_{C(\overline D_s)}\le s e^{2L_1s}\left(d_1\|\varphi_1-\rho_1\|_{C(
 \overline D_s)}+L_1\|p_2\|_{C(\overline D_s)}\right).\]
Similarly,
 \[\|p_2\|_{C(\overline D_s)}\le s e^{2L_2s}\left(d_2\|\varphi_2-\rho_2\|_{C(
 \overline D_s)}+L_2\|p_1\|_{C(\overline D_s)}\right),\]
where $L_2=d_2+L(C)$. Set $L=\max\{L_1, L_2\}$. Then
 \begin{align*}
 \|\Gamma_s \varphi-\Gamma_s\rho\|_{C(\overline D_s)}
 \le\frac 12\left(\|\varphi_1-\rho_1\|_{C(\overline D_s)}+\|\varphi_2-\rho_2\|_{C(\overline D_s)}\right),~ \ \forall~s\in (0, \sigma]
 \end{align*}
provided that $\sigma\in (0,\delta]$ satisfies
 \[\sigma Le^{2L\sigma}\leq 1/2, \ \ \ \sigma d_ie^{2L\sigma}\le 1/4,
 \ \ i=1,\,2.\]
For such $s$ we may now apply the Contraction Mapping Theorem to conclude that $\Gamma_s$ has a unique fixed point $W$ in $\mathbf X^s_C$. It follows that $w=W$ solves \eqref{2.3} for $0<t\leq s$.

If we can show that any solution $w$ of \eqref{2.3} satisfies $0\leq w_1, w_2\leq C$ in $D_s$ then $w$ must coincides with the unique fixed point $W$ of $\Gamma_s$ in $\mathbf X^s_C$. We next prove such an estimate for $(w_1,w_2)$. Note that $w_1,w_2 \geq 0$ already follows from \eqref{2.5}. It is enough to show $w_1,w_2\leq C$. We actually prove the following stronger inequality
\begin{equation}\label{2.7}
 w_1(t,x)\leq A_1, \ w_2(t,x)\le A_2  ,\ \ \forall \ g(s)\le x\le h(s), \ t_x\le t\le s.
\end{equation}
We only prove $w_1(t,x)\leq A_1$ since $w_2(t,x)\le A_2$ can be shown by the same way. It suffices to show that the above inequality holds with $A_1$ replaced by $A_1+\ep$ for any given $\ep>0$. Suppose this is not true. Due to
$w_1(t_x, x)=\tilde u_{10}(x)\leq \|u_{10}\|_\infty<A^\ep:=A_1+\ep$, there exist $x_0\in (g(s), h(s))$ and $t_0\in (t_{x_0}, s]$ such that
 \[w_1(t_0, x_0)=A^\ep, \ \ \ 0\leq w_1(t,x)<A^\ep \ \mbox{ for} \
 g(t_0)\le x\le h(t_0), \ t_x\le t< t_0.\]
It follows that $w_{1t}(t_0, x_0)\ge 0$ and $f_1(t_0, x_0, w(t_0, x_0))\leq 0$ due to $w_1(t_0, x_0)=A^\ep>A_1$ and $w_2(t_0, x_0)\ge 0$.
Hence from the first equation of \eqref{2.3} we obtain
 \[0\leq w_{1t}(t_0,x_0)\le d_1\int_{g(t_0)}^{h(t_0)}J_1(x_0-y)w_1(t_0,y)\dy-d_1w_1(t_0, x_0).\]
Since $w_1(t_0,g(t_0))=w_1(t_0, h(t_0))=0$, we have $w_1(t_0,y)<A^\ep$ for $y\in (g(t_0), h(t_0))$ but close to the boundary of this interval. It follows that
 \[d_1A^\ep=d_1w_1(t_0,x_0)\le  d_1\int_{g(t_0)}^{h(t_0)}J_1(x_0-y) w_1(t_0,y)\dy<d_1A^\ep\int_{g(t_0)}^{h(t_0)}J_1(x_0-y)\dy\leq d_1A^\ep.\]
This contradiction proves \eqref{2.7}.
 Thus $(w_1, w_2)$ satisfies the wanted inequality and hence coincides with the unique fixed point
of $\Gamma_s$ in $\mathbf X^s_C$. We have now proved the fact that for every $s\in (0,\sigma]$, $\Gamma_s$ has a unique fixed point in $\mathbb X_{U_0}^s$.
\smallskip

{\bf Step 3:} {\it Completion of the proof.} From Step 2 we know that \eqref{2.3} has a unique solution $w$ defined for $t\in [0, \sigma]$ and $w$ satisfies \eqref{2.7} with $s=\sigma$. Note that
 \begin{align*}
 \max\left\{\max_{[g(\sigma),h(\sigma)]}w_1(\sigma,x), \ k\right\}\le&\;
 \max\{A_1,\ k\}=A_1, \\
 \max\left\{\max_{[g(\sigma),h(\sigma)]}w_2(\sigma,x), \ \Theta(A_1)\right\}\le&\;
 \max\{A_2,\ \Theta(A_1)\}=A_2.
 \end{align*}
Hence we may apply Step 2 to \eqref{2.3} but with the initial time $t=0$ replaced by $t=\sigma$ to conclude that
 the unique solution can be extended to a slightly larger domain $D_{g,h}^{\ol\sigma}$. Moreover, by \eqref{2.7} and the definition of $\sigma$ in Step 2, we see that $\ol\sigma$ depends only on $d_i$ and $A_i$, and it can take any value in $(0, 2\sigma]$. Furthermore, from the above proof of \eqref{2.7} we easily see that the extended
solution $w$ satisfies \eqref{2.7} in $D_{g,h}^{\ol\sigma}$. Thus the extension can be repeated. By repeating this process finitely many times, the solution $w$ of \eqref{2.3} will be uniquely extended to $D^T_{g, h}$. As explained above, now
\eqref{2.7} holds with $s=T$, and hence to prove that $w$ satisfies \eqref{2.1}, it only remains
to show $w_1>0$, $w_2>0$ in $D^T_{g, h}$. Recall $w_1, w_2\ge 0$. Using the conditions {\rm\bf (f), (f1), (f2)} and the conclusion \eqref{2.7} we may write $f_1(t, x, w)=b_1(t,x)w_1$ and $f_2(t, x, w)=b_2(t,x)w_2$ with $b_i\in L^\infty(D^T_{g, h})$. Then Lemma \ref{l2.2} gives the desired result.
\end{proof}

\begin{proof}[Proof of Theorem \ref{th2.1}]\, By Lemma \ref{l2.3},
for any $T>0$ and $(g, h)\in\mathbb{G}_{h_0}^T\times\mathbb H_{h_0}^T$, we can find a
unique $w_{g,h}=(w_{1,g,h},\,w_{2,g,h})\in\mathbb{X}^T_{u_0, g, h}$ that solves (\ref{2.3}) { and satisfies \eqref{2.1}.}

Using such a $w_{g,h}$, we define the mapping ${\mathcal F}$ by ${\mathcal F} (g,h)=
\bbb(\tilde g,\,\tilde h\bbb)$, where
 \begin{equation*}\begin{aligned}
&\tilde h(t)=h_0+\sum_{i=1}^2\mu_i\int_0^t\int_{g(\tau)}^{h(\tau)}\int_{h(\tau)
}^{\infty}J_i(x-y)w_{i,g,h}(\tau,x)\dy\dx{\rm d}\tau,\\
&\tilde g(t)=-h_0-\sum_{i=1}^2\mu_i\int_0^t\int_{g(\tau)}^{h(\tau)}\int_{-\infty
}^{g(\tau)}J_i(x-y)w_{i,g,h}(\tau,x)\dy\dx{\rm d}\tau
 \end{aligned}
\end{equation*}
for $0<t\leq T$. To simplify notations, we will write
 \[\mathbb G_T=\mathbb G_{h_0}^T,\ \ \ \mathbb H_T=\mathbb H_{h_0}^T,\ \ \ D_T=D^T_{g,h}, \ \ \
\mathbb X_T=\mathbb{X}^T_{U_{0}, g, h}.\]

To prove this theorem, we first show that if $T$ is small enough, then ${\mathcal F}$ maps a suitable closed subset $\Sigma_T$ of $\mathbb{G}_T\times\mathbb{H}_T$ into itself,
and  is a contraction mapping. This clearly implies that ${\mathcal F}$ has a unique fixed point in $\Sigma_T$, which gives a solution $(w_{g,h}, g, h)$ of \eqref{1.2} defined for $t\in (0, T]$. Then we prove that any solution $(u_1, u_2, g, h)$ of \eqref{1.2} with $(g,h)\in \mathbb{G}_T\times\mathbb{H}_T$ must satisfy $(g,h)\in \Sigma_T$, and hence $(g,h)$ must coincide with the unique fixed point of ${\mathcal F}$ in $\Sigma_T$, which then implies that the solution $(u_1, u_2, g, h)$ of \eqref{1.2} is unique. Finally  we extend this unique local solution to a global one. This plan will be carried out in several steps.

\sk{\bf Step 1:} {\it Properties of $(\tilde g, \tilde h)$ and a closed subset of $\mathbb{G}_T\times\mathbb{H}_T$.} Let $(g,h)\in \mathbb{G}_T\times\mathbb{H}_T$.  Then $\tilde g, \tilde h\in C^1([0, T])$ and for $0<t\le T$,
 \be\left\{\begin{aligned}
&\tilde h'(t)=\sum_{i=1}^2\mu_i\int_{g(t)}^{h(t)}\int_{h(t)
}^{\infty}J_i(x-y)w_{i,g,h}(\tau,x)\dy\dx,\\
&\tilde g'(t)=-\sum_{i=1}^2\mu_i\int_{g(t)}^{h(t)}\int_{-\infty
}^{g(t)}J_i(x-y)w_{i,g,h}(\tau,x)\dy\dx.
 \end{aligned}\right.
 \label{2.8}\ee
These facts imply $(\tilde g, \tilde h)\in  \mathbb{G}_T\times\mathbb{H}_T$. To show that ${\mathcal F}$ is a contraction, we need some further properties of $\tilde g$ and $\tilde h$ to be used in choosing a suitable closed subset of  $\mathbb{G}_T\times\mathbb{H}_T$, which
is invariant under ${\mathcal F}$, and on which ${\mathcal F}$ is a contraction mapping.

Denote $w_i=w_{i,g,h}$ to simplify the notations. { Since $(w_1, w_2)$ solves (\ref{2.3}) and satisfies \eqref{2.1}, we obtain by using {\bf (f), (f1), (f2)} that}
 \bess\left\{\begin{aligned}
 &w_{it}(t,x)\ge-d_iw_i(t,x)-L(A_1,A_2)w_i(t,x),&& 0<t\le T,~g(t)<x<h(t),\\
 &w_i=0, \ && 0\le t\le T,~ x=g(t),\,h(t),\\
 &w_i(0,x)=u_{i0}(x), &&|x|\le h_0.
  \end{aligned}\right. \eess
It follows that
 \bes
  w_i(t,x)\ge e^{-[d_i+L(A_1,A_2)]t}u_{i0}(x)\ge e^{-[d_i+L(A_1,A_2)]T}u_{i0}(x),\ \
  t\in (0, T],\ |x|\le h_0.\label{2.9}
 \ees
By the condition \textbf{(J)}, there exist constants $\ep_0\in (0, h_0/4)$ and $\delta_0>0$ such that
  \bes\label{2.10}
 J_i(x-y)\ge\delta_0\ \ \ \mbox{for}\ \ |x-y|\le\ep_0, \ \ \ i=1,2.
 \ees
Using \eqref{2.8} we easily see
 \[[\tilde h(t)-\tilde g(t)]'\leq (\mu_1 A_1+\mu_2 A_2)[h(t)-g(t)],\ \ \forall\ t\in [0, T].\]
We can choose $0<T\ll 1$, depending on $\mu_i, A_i, h_0, \ep_0$, such that $h(T)-g(T)\leq 2h_0+\ep_0/4$ and
   \begin{align*}
   \tilde h(t)-\tilde g(t)\leq&\; 2h_0 +T(\mu_1 A_1+\mu_2 A_2)(2h_0+\ep_0/4)\leq 2h_0
  +\ep_0/4,\ \ \forall\ t\in [0, T],\\
 h(t)\in&\;[h_0,\,h_0+\ep_0/4],\ \ g(t)\in [-h_0-\ep_0/4,\,-h_0],\ \ \forall\ t\in [0, T].
  \end{align*}
Combining this with \eqref{2.9} and \eqref{2.10} we obtain that, for $t\in (0, T]$,
 \begin{align*}
 \sum_{i=1}^2\mu_i\int_{g(t)}^{h(t)}\!\int_{h(t)}^{\infty}\!J_i(x-y)w_i(t,x)\dy\dx
\ge&\;\sum_{i=1}^2\mu_i\int_{h(t)-\frac{\ep_0}{2}}^{h(t)}\!\int_{
h(t)}^{h(t)+\frac{\ep_0}{2}}\!J_i(x-y)w_i(t,x)\dy\dx\\
\ge&\;\sum_{i=1}^2\mu_i e^{-LT}\int_{h_0-\frac{\ep_0}4}^{h_0}\!\int_{
h_0+\frac{\ep_0}4}^{h_0+\frac{\ep_0}{2}}\!J_i(x-y)e^{-d_iT}u_{i0}(x)\dy\dx\\
\ge&\;\frac {\ep_0}4\delta_0e^{-(d_1+d_2+L)T}\sum_{i=1}^2\mu_i\int_{h_0-\frac{
\ep_0}4}^{h_0}u_{i0}(x)\dx\\
=:&\;\alpha_1\mu_1+\alpha_2\mu_2,
 \end{align*}
where $L=L(A_1,A_2)$, $\alpha_1$, $\alpha_2$ are positive constants depending only on $J_i,\,f_i$ and $U_0$. Thus, for sufficiently small $T_0=T(\mu_1, \mu_2 , A_1, A_2, h_0, \ep_0)>0$,
 \bes \label{2.11}
 \tilde h'(t)\geq \alpha_1\mu_1+\alpha_2\mu_2 :=h_*>0, \ \ t\in [0, T_0].
 \ees
Similarly,
 \bes\label{2.12}
 \tilde g'(t)\leq -(\tilde \alpha_1\mu_1+\tilde \alpha_2\mu_2):=g_*<0, \ \ t\in [0, T_0]
 \ees
for some positive constants $\tilde \alpha_1$ and $\tilde \alpha_2$ depending only on $J_i,\,f_i$ and $U_0$.

For $0<T\le T_0$, we define
 \begin{align*}
\Sigma_T:=&&\Big\{(g,h)\in\mathbb G_T\times\mathbb H_T: \frac{g(t_2)-g(t_1)}{t_2-t_1}\leq g_*, \ \frac{h(t_2)-h(t_1)}{t_2-t_1}\geq h_* \mbox{ for } 0\leq t_1<t_2\leq T, \\
&&  \mbox{ and } \ h(t)-g(t)\leq 2h_0+\frac{\ep_0}4 \mbox{ for } 0\leq t\leq T\Big\}.
 \end{align*}
The above analysis shows that ${\mathcal F}(\Sigma_T)\subset\Sigma_T$.

\sk {\bf Step 2:} {\it ${\mathcal F}$ is a contraction mapping on $\Sigma_T$ for $0<T\ll 1$.}
For $(g_j,h_j)\in\Sigma_T$, $j=1,2$, we set
  \begin{align*}
  \Omega_T=D_{g_1, h_1}^T&\cup D_{g_2, h_2}^T, \ \ w_1^j=w_{1,g_j,h_j}, \ \ w_2^j=w_{2,g_j,h_j}, \ \ {\mathcal F}\left(g_j,h_j\right)=\bbb(\tilde g_j,\tilde h_j\bbb),\\
\hat w_j=w_{j}^1&-w_{j}^2, \ \ \hat g=g_1-g_2, \ \ \hat h=h_1-h_2, \ \
     \tilde g=\tilde g_1-\tilde g_2, \ \ \tilde h=\tilde h_1-\tilde h_2.
     \end{align*}
Make the zero extension of $w_i^j$ in $\big([0,T]\times\mathbb R\big)\setminus D^T_{g_j, h_j}$. It then follows that
  \begin{align*}
 |\tilde h'(t)|
\le&~\sum_{i=1}^2\mu_i\left|\int_{g_1(\tau)}^{h_1(\tau)}\!\!\int_{h_1(\tau)}^{\infty}
J_i(x-y)w_i^1(\tau,x)\dy\dx-\int_{g_2(\tau)}^{h_2(\tau)}\!\!
\int_{h_2(\tau)}^{\infty}J_i(x-y)w_i^2(\tau,x)\dy\dx\right|\\
\le&~\sum_{i=1}^2\mu_i\int_{g_1(\tau)}^{h_1(\tau)}\!\int_{h_1(\tau)}^{\infty}\!\!
 J_i(x-y)|\hat w_i(\tau,x)|\dy\dx\\
 &+\sum_{i=1}^2\mu_i\bigg|\left(\int_{g_1(\tau)}^{g_2(\tau)}\!\!\int_{h_1(\tau)}^{\infty}
 +\int_{h_2(\tau)}^{h_1(\tau)}\!\!\int_{h_1(\tau)}^{\infty}
 +\int_{g_2(\tau)}^{h_2(\tau)}\!\!\int_{h_1(\tau)}^{h_2(\tau)}\right)
 J_i(x-y)w_i^2(\tau,x)\dy\dx\bigg|\\
\le&~\sum_{i=1}^2\mu_i\left(3h_0\|\hat w_i\|_{C(\overline\Omega_T)}
  +\|\hat g\|_{C([0,T])}A_i+\left(A_i+3h_0 A_i\|J_i\|_\infty\right)\|\hat h\|_{C([0,T])}\right),\end{align*}
and so
 \begin{align*}
|\tilde h(t)|\le C_0T\left(\|\hat w_1,\,\hat w_2\|_{C(\overline\Omega_T)}+\|\hat g,\,\hat h\|_{C([0,T])}\right),\ \ 0<t\le T,
 \end{align*}
where $C_0$ depends only on $h_0$, $\mu_i ,\,A_i$ and $J_i$. Similarly,
 \[|\tilde g(t)|\le C_0T\left(\|\hat w_1,\,\hat w_2\|_{C(\overline\Omega_T)}+\|\hat g,\,\hat h\|_{C([0,T])}\right),\ \ 0<t\le T.\]
Therefore,
  \be \begin{aligned}
\|\tilde g,\,\tilde h\|_{C([0,T])}
\le CT\left(\|\hat w_1,\,\hat w_2\|_{C(\overline\Omega_T)}+\|\hat g,\,\hat h\|_{C([0,T])}\right).
 \end{aligned}\label{2.13}\ee

Next, we  estimate $\|\hat w_1,\, \hat w_2\|_{C(\overline\Omega_T)}$. Fix $(s,x)\in\Omega_T$. We now estimate $|\hat w_1(s,x)|$ and $|\hat w_2(s,x)|$ in all the possible cases.

\sk\underline{Case 1}: $x\in(g_1(s),h_1(s))\setminus(g_2(s),h_2(s))$.
In such case, either $g_1(s)<x\le g_2(s)$ or $h_2(s)\le x<h_1(s)$, and $w_1^2(s,x)=w_2^2(s,x)=0$.

When $h_2(s)\le x<h_1(s)$, there exists $0<s_1<s$ such that $x=h_1(s_1)$, and so $h_1(s)>h_1(s_1)=x\ge h_2(s)$. Clearly, $g_1(t)<h_1(s_1)=x\le h_1(t)$ for all $t\in[s_1, s]$. By integrating the  equation
satisfied by $w_1^1$ from $s_1$ to $s$ we obtain
  \begin{align*}
|\hat w_1(s, x)|&=w_1^1(s,x)=\int_{s_1}^{s}\left(d_1\int_{g_1(t)}^{h_1(t)}J_1(x-y)w_1^1(t, y)\dy-d_1w_1^1+f_1(t, x, w_1^1, w_2^1)\right)\dt\\
&\leq (s-s_1)\big[d_1+L(A_1,A_2)\big]A_1\\
&\leq h_*^{-1}\big[h_1(s)-h_1(s_1)\big]\big[d_1+L(A_1,A_2)\big]A_1\\
&\leq h_*^{-1}\big[h_1(s)-h_2(s)\big]\big[d_1+L(A_1,A_2)\big]A_1 \\
&\leq C_4\|h_1-h_2\|_{C([0,s])}=C_4\|\hat h\|_{C([0,T])}.
 \end{align*}
When $g_1(s)<x\le g_2(s)$, we can analogously obtain
 \[|\hat w_1(s, x)|=w_1^1(s,x)\leq C_4\|\hat g\|_{C([0,T])}.\]
 Hence
 \[|\hat w_1(s, x)|\leq C_4\|\hat g,\,\hat h\|_{C([0,T])}.\]

Similarly we can show $|\hat w_2(s, x)|=w_2^1(s,x)\leq C_4\|\hat g,\,\hat h\|_{C([0,T])}$. Thus, in this case,
 \bes
 |\hat w_1(s, x)|,\,|\hat w_2(s, x)|\leq C_4\|\hat g,\,\hat h\|_{C([0,T])}.
 \lbl{2.14}\ees

\underline{Case 2}: $x\in(g_2(s),h_2(s))\setminus(g_1(s),h_1(s))$. Similar to  case 1 we have
$|\hat w_1(s, x)|=w_1^2(s,x)\leq C_4\|\hat g, \ \hat h\|_{C([0,T])}$ and $|\hat w_2(s, x)|=w_2^2(s,x)\leq C_4\|\hat g, \ \hat h\|_{C([0,T])}$. Thus \eqref{2.14} holds in this case as well.

\sk\underline{Case 3}: $x\not\in(g_2(s),h_2(s))\cup(g_1(s),h_1(s))$. In this case clearly $\hat w_1(s,x)=\hat w_2(s,x)=0$ and hence \eqref{2.14} holds trivially.

\sk\underline{Case 4}: $x\in(g_1(s), h_1(s))\cap(g_2(s),h_2(s))$. If $x\in (g_1(t),h_1(t))\cap(g_2(t),h_2(t))$ for all $0<t<s$, that is, $x\in [-h_0, h_0]$, it then follows that
 \bes
\hat w_{1t}(t,x)&=&d_1\displaystyle\int_{g_1(t)}^{h_1(t)}J_1(x-y)\hat w_1(t,y)\dy+d_1\left\{\int_{g_1(t)}^{g_2(t)}
+\int_{h_2(t)}^{h_1(t)}\right\}J_1(x-y)w_1^2(t,y)\dy\nonumber\\[1mm]
&&-d_1\hat w_1(t,x)+f_1(t,x,w_1^1,w_2^1)-f_1(t,x,w_1^2,w_2^2).
 \lbl{2.15}
 \ees
Note that $\hat w_1(0,x)=0$, $0<w_1^i\le A_1$, $0<w_2^i\le A_2$ and
\[|f_1(t,x,w_1^1,w_2^2)-f_1(t,x,w_1^2,w_2^2)|\le L(|w_1^1-w_1^2|+|w_2^1-w_2^2|)=L(|\hat w_1|+|\hat w_2|),\]
where $L=L(A_1,A_2)$. Integrating \eqref{2.15} from $0$ to $s$ we have
 \begin{align*}
 |\hat w_1(s,x)|\le\left(2d_1\|\hat w_1\|_{C(\Omega_T)}+d_1A_1\|J_1\|_\infty\|\hat g,\,\hat h\|_{C([0,T])}
 +L\|\hat w_1,\,\hat w_2\|_{C(\Omega_T)}\right)T.
 \end{align*}
Similarly,
 \begin{align*}
 |\hat w_2(s,x)|\le\left(2d_2\|\hat w_2\|_{C(\Omega_T)}+d_2A_2\|J_2\|_\infty\|\hat g,\,\hat h\|_{C([0,T])}
 +L\|w_1,\,w_2\|_{C(\Omega_T)}\right)T.
 \end{align*}

If there exists $0<t<s$ such that $x\not\in(g_1(t),h_1(t))\cap(g_2(t),h_2(t))$, then we can choose the largest $t_0\in(0,t)$ such that
 \bes
 x\in(g_1(t),h_1(t))\cap(g_2(t),h_2(t)), \ \ \forall \ t_0<t\le s,
 \lbl{2.16}\ees
and
\bess
 x\in(g_1(t_0),h_1(t_0))\setminus(g_2(t_0),h_2(t_0)), \ \ \mbox{or} \ \
 x\in(g_2(t_0),h_2(t_0))\setminus(g_1(t_0),h_1(t_0)).
 \eess
Using the conclusions of Case 1 and Case 2 we have $|\hat w_1(t_0, x)|\leq C_4\|\hat g,\,\hat h\|_{C([0,T])}$.
In view of \eqref{2.16}, it is clear that \eqref{2.15} holds for all $t_0<t\le s$. Integrating \eqref{2.15} from $t_0$ to $s$ we obtain
 \begin{align*}
 |\hat w_1(s,x)|\le&\;|\hat w_1(t_0, x)|+\left(2d_1\|\hat w_1\|_{C(\Omega_T)}+d_1A_1\|J_1\|_\infty\|\hat g,\,\hat h\|_{C([0,T])}
 +L\|\hat w_1,\,\hat w_2\|_{C(\Omega_T)}\right)(s-t_0)\\
 \le&\;C_4\|\hat g,\,\hat h\|_{C([0,T])}+C_5\left(\|\hat g,\,\hat h\|_{C([0,T])}
 +\|\hat w_1,\,\hat w_2\|_{C(\Omega_T)}\right)T\\
 =:&\;C_6\|\hat g,\,\hat h\|_{C([0,T])}+C_5T\|\hat w_1,\,\hat w_2\|_{C(\Omega_T)}.
 \end{align*}
In the same way one has
 \begin{align*}
 |\hat w_2(s,x)|\le C_6\|\hat g,\,\hat h\|_{C([0,T])}+C_5T\|\hat w_1,\,\hat w_2\|_{C(\Omega_T)}.
 \end{align*}

Summarizing the above discussions, we obtain
  \bess
 \|\hat w_1,\,\hat w_2\|_{C(\Omega_T)}\le C'\|\hat g, \ \hat h\|_{C([0,T])}+C'T\|\hat w_1,\,\hat w_2\|_{C(\Omega_T)}
 \le 2C'\|\hat g, \ \hat h\|_{C([0,T])}
 \eess
if $C'T<1/2$. This combined with \eqref{2.13} yields
  \begin{align*}
\|\tilde g,\,\tilde h\|_{C([0,T])}
\le C(2C'+1)T\|\hat g,\,\hat h\|_{C([0,T])}\le\frac 12\|\hat g,\,\hat h\|_{C([0,T])}
 \end{align*}
if $C(2C'+1)T\le 1/2$. This shows that ${\mathcal F}$ is a contraction mapping on $\Sigma_{T}$.

\medskip
{\bf Step 3:} {\it Local existence and uniqueness.} By Step 2 and the Contraction Mapping Theorem we know that \eqref{1.2} has a solution $(u_1, u_2, g, h)$ defined for $t\in (0, T]$. If we can show that $(g,h)\in\Sigma_{T}$ holds for any solution $(u_1, u_2, g, h)$
of \eqref{1.2} defined for $t\in (0, T]$, then $(g,h)$  must coincide with the unique fixed point of ${\mathcal F}$ in $\Sigma_{T}$ and the uniqueness of the local solution $(u_1, u_2, g, h)$ to \eqref{1.2} would follow.

Let $(u_1, u_2, g, h)$ be an arbitrary solution of \eqref{1.2} defined for $t\in (0, T]$. Then
  \[\begin{aligned}
 &h'(t)=\sum_{i=1}^2\mu_i\int_{g(t)}^{h(t)}\int_{h(t)}^{\infty}J_i(x-y)u_i(t,x)\dy\dx,\\
 &g'(t)=-\sum_{i=1}^2\mu_i\int_{g(t)}^{h(t)}\int_{-\infty}^{g(t)}J_i(x-y)u_i(t,x)\dy\dx.
 \end{aligned}\]
In view of Lemma \ref{l2.3}, $0<u_1\leq A_1,\,0<u_2\leq A_2$ in $D^T_{g, h}$. Thus
 \begin{align*}
 h'(t)-g'(t)=\sum_{i=1}^2\mu_i\int_{g(t)}^{h(t)}\!\left(\int_{h(t)}^{\infty}
 +\int_{-\infty}^{g(t)}\right)J_i(x-y)u_i(t,x)\dy\dx
 \leq\sum_{i=1}^2\mu_iA_i[h(t)-g(t)], \end{align*}
which implies
 \bes\label{2.17}
 h(t)-g(t)\leq 2h_0 e^{(\mu_1A_1+\mu_2 A_2)t}, \ \ \ t\in (0, T].
 \ees
Shrink $T$ so that $2h_0e^{(\mu_1A_1+\mu_2 A_2)T}\leq 2h_0+{\ep_0}/4$, then
$h(t)-g(t)\leq 2h_0+\ep_0/4$ on $[0, T]$. Moreover, the proof of \eqref{2.11} and \eqref{2.12} gives
$h'(t)\geq h_*$ and $g'(t)\leq g_*$ in $(0, T]$. Thus $(g,h)\in\Sigma_{T}$ as we required.

\smallskip
{\bf Step 4:} {\it Global existence and uniqueness.} By Step 3, we see that the problem \eqref{1.2} has a unique solution $(u_1, u_2, g, h)$ for some time interval $(0, T]$. Moreover, for any fixed $s\in(0, T)$, there hold $u_1(s,x)>0$, $u_2(s,x)>0$ in $(g(s), h(s))$, and $u_i(s,\cdot)$ ($i=1,2$) are  continuous on $[g(s), h(s)]$.
This implies that we can treat $(u_1(s,\cdot), u_2(s,\cdot))$ as an initial function and use Step 3
to extend the solution from $t=s$ to some $T'\geq T$. Suppose that $(0, T_0)$ is the maximal existence interval of $(u_1, u_2, g, h)$ obtained by such an extension process. We show that $T_0=\infty$. Otherwise $T_0\in(0, \infty)$ and we are going to derive a contradiction.

Firstly, \eqref{2.17} holds for $t\in (0, T_0)$. Since $h(t)$ and $g(t)$ are
monotone in $[0, T_0)$, we may define
  \[h(t_0):=\lim_{t\to T_0} h(t),\ \ \ g(T_0):=\lim_{t\to T_0} g(t) \ \ \mbox{ with }
   \ h({T_0})-g(T_0)\leq 2h_0e^{(\mu_1A_1+\mu_2 A_2)T_0}.\]
The free boundary conditions in \eqref{1.2}, together with $0\le u_1\leq A_1$ and $0\le u_2\leq A_2$ indicate that $h',\,g'\in L^\infty([0, T_0))$ and hence $g, h\in C([0,T_0])$ with $g(T_0)$, $h(T_0)$
defined as above. It follows that the right-hand side of the first
equation in \eqref{1.2} belongs to $L^\infty(D^{T_0}_{g, h})$, this implies
$u_{i,t}\in L^\infty(D^{T_0}_{g, h})$. Thus for each $x\in (g(T_0), h(t_0))$ and $i=1,2$, the limit
$u_i(T_0,x):=\dd\lim_{t\nearrow T_0}u_i(t,x)$ exists, and $u_i(\cdot, x)$ is continuous at $t=T_0$. We may now view $(u_1(t,x), u_2(t,x))$ as the unique solution of the ODE problem in Step 1 of the proof of Lemma \ref{l2.3} (with $\varphi=(u_1, u_2)$), which is defined over $[t_x, T_0]$. Since $t_x$, $J_i(x,y)$, and
$f_i(t, x, u_1, u_2)$ are all continuous in $x$, by the continuous dependence of the ODE solution to
the initial function and the parameters in the equation, we see that $u_i(t,x)$ ($i=1,2$) are
continuous in $D^{T_0}_{g, h}$. By assumption, $u_1,\,u_2\in C(\overline D^s_{g, h})$ for any
$s\in (0, T_0)$. To show this also holds with $s=T_0$, it remains to show that
$u_i(t,x)\to 0$ as $(t,x)\to (T_0, g(T_0))$ and as $(t,x)\to (T_0, h(t_0))$ from $D^{T_0}_{g, h}$. We only prove $u_1(t,x)\to 0$ as $(t,x)\to (T_0, g(T_0))$ from $D^{T_0}_{g, h}$
because of the other cases can be shown similarly. We note that $x\searrow g(T_0)$ implies
$t_x\nearrow T_0$, and so
 \begin{align*}
|u_1(t,x)|&=\left|\int_{t_x}^t \left(d_1\int_{g(t)}^{h(t)}J_1(x,y)u_1(\tau, y)\dy-d_1 u_1(\tau, x)
+f_1(\tau, x, u_1(\tau,x), u_2(\tau, x))\right){\rm d}\tau\right|\\[1mm]
&\leq (t-t_x)\big[2d_1+L(A_1, A_2)\big]A_1\\[1mm]
&\to 0 \ \ \mbox{ as }\ D^{T_0}_{g, h}\ni(t,x)\to(T_0, g(T_0)).
\end{align*}

Thus we have shown that $u_i\in C(\overline D^{T_0}_{g, h})$ and $(u_1, u_2, g, h)$ satisfies
\eqref{1.2} for $t\in (0, T_0]$. Writing $f_i(t,\,x,\,u_1(t,x), u_2(t,x))=b_i(t,x)u_i(t,x)$ with $b_i\in L^\infty(D^{T_0}_{g, h})$, and using
Lemma \ref{l2.2} we have $u_i(T_0, x)>0$ for $x\in (g(T_0), h(t_0))$. Thus we can
regard $(u_1(T_0, \cdot), u_2(T_0, x))$ as an initial function and apply Step 3 to conclude that the solution of \eqref{1.2} can be extended to some $(0, \tilde T)$ with $\tilde T>T_0$. This
contradicts the definition of $T_0$. Therefore we must have $T_0=\infty$.

From the above proof we see that \eqref{2.1} and \eqref{2.2} hold, and the theorem is proved.
\end{proof}

\section{Spreading and vanishing }
\setcounter{equation}{0} {\setlength\arraycolsep{2pt}

In view of \eqref{2.2} we can define
 \[
 \lim_{t\to\infty}g(t)=g_\infty\in[-\infty,-h_0), \ \ \ \lim_{t\to\infty}h(t)=h_\infty\in(h_0,\infty].
 \]
 Clearly we have either
 \[
 \mbox{(i)\, $h_\infty-g_\infty<\infty$, \ \ or \ (ii)\, $h_\infty-g_\infty=\infty$. }
 \]
 We will call (i) the vanishing case, and call (ii) the spreading case.

 The main purpose of this section is to determine when (i) or (ii) can occur, and to determine the long-time profile of $(u_1, u_2)$ if (i) or (ii) happens.
 It turns out that these are highly nontrivial tasks as many techniques worked in the corresponding local diffusion cases are not applicable anymore, and those worked in the one species nonlocal diffusion problem with free boundary in \cite{CDLL} are also lacking for treating the current two species situation. In subsection 3.1 below, we introduce some new techniques which are enough to treat the Lotka-Volterra cases \eqref{1.4} and \eqref{1.5}. In subsection 3.2, we will further restrict the growth function classes in order to determine the long-time profile of $(u_1, u_2)$.

\subsection{Criteria for vanishing and spreading}
The following two simple lemmas provide some key ingredients for analysing the vanishing phenomenon.

\begin{lem}\lbl{l3.1}\,Let the condition {\bf (J)} hold for the kernel functions $J_1, J_2$, and $\beta_1,\beta_2>0$ be  constants.
Suppose that $g,\,h\in C^1([0,\infty))$, $g(0)<h(0)$, $g'(t)\le 0$, $h'(t)\ge 0$ and $w_i,\,w_{it}\in C(D_{g,h}^\infty)\cap L^\infty(D_{g,h}^\infty)$ for $i=1,2$, where $D_{g,h}^\infty=\left\{t>0,~g(t)<x<h(t)\right\}$.
If $(w_1, w_2,\,g,\,h)$ satisfies
  \bes\left\{\begin{aligned}
 &h'(t)=\sum_{i=1}^2\beta_i\int_{g(t)}^{h(t)}\int_{h(t)}^{\infty}J_i(x-y)w_i(t,x)\dy\dx,
 &&t\ge 0,\\
&g'(t)= -\sum_{i=1}^2\beta_i\int_{g(t)}^{h(t)}\int_{-\infty}^{g(t)}J_i(x-y)w_i(t,x)\dy\dx,
&&t\ge 0, \end{aligned}\right.
 \label{3.1}\ees
 and
 \bes
  \lim_{t\to\infty}h(t)-\lim_{t\to\infty}g(t)<\infty,
 \label{3.2}
 \ees
then
  \begin{align*}
 \lim_{t\to\infty}g'(t)=\lim_{t\to\infty}h'(t)=0.
 \lbl{zz7.25}\end{align*}
\end{lem}

\begin{proof}\, Set $\dd\lim_{t\to\infty}h(t)=h_\infty$, $\dd\lim_{t\to\infty}g(t)=g_\infty$. The condition \eqref{3.2} implies $-\infty<g_\infty<h_\infty<\infty$. Take $M>0$ such that
$\beta_i|w_i|\le M$, $\beta_i|w_{it}|\le M$ in $D_{g,h}^\infty$ for $i=1,2$. It then follows from \eqref{3.1} that $g'(t)$ and $h'(t)$ are bounded. For any given $t,s>0$, we have
 \begin{align*}
 {h'(t)-h'(s)}=&\;\int_{g(t)}^{h(t)}\int_{h(t)}^{\infty}\sum_{i=1}^2\beta_iJ_i(x-y)w_i(t,x)\dy\dx
 -\int_{g(s)}^{h(s)}\int_{h(s)}^{\infty}\sum_{i=1}^2\beta_iJ_i(x-y)w_i(s,x)\dy\dx\\
 =&\;\left\{\int_{g(t)}^{g(s)}+\int_{h(s)}^{h(t)}\right\}\int_{h(t)}^{\infty}\sum_{i=1}^2\beta_iJ_i(x-y)w_i(t,x)\dy\dx\\
 &-\int_{g(s)}^{h(s)}\int_{h(s)}^{h(t)}\sum_{i=1}^2\beta_iJ_i(x-y)w_i(s,x)\dy\dx\\
 &+\;\int_{g(s)}^{h(s)}\int_{h(t)}^{\infty}\sum_{i=1}^2\beta_iJ_i(x-y)\big[w_i(t,x)-w_i(s,x)\big]\dy\dx\\
 =&\;\left\{\int_{g(t)}^{g(s)}+\int_{h(s)}^{h(t)}\right\}\int_{h(t)}^{\infty}\sum_{i=1}^2\beta_iJ_i(x-y)w_i(t,x)\dy\dx\\
 &-\int_{h(s)}^{h(t)}\int_{g(s)}^{h(s)}\sum_{i=1}^2\beta_iJ_i(x-y)w_i(s,x)\dx\dy\\
 &+\int_{g(s)}^{h(s)}\int_{h(t)}^{\infty}\sum_{i=1}^2\beta_iJ_i(x-y)w_{it}(\tau_i(x),x)(t-s)\dy\dx,
 \end{align*}
 where $\tau_i(x)$ is a number lying between $s$ and $t$.
Therefore
 \begin{align*}
 |h'(t)-h'(s)|\le&\; 2M\big(|g(t)-g(s)|+|h(t)-h(s)|\big)+2M|h(t)-h(s)|+2M(h_\infty-g_\infty)|t-s|\\[1mm]
 \le&\;2M\big(3M'+h_\infty-g_\infty\big)|t-s|,
 \end{align*}
where $M'=\|g'\|_\infty+\|h'\|_\infty$. This shows that $h'(t)$ is Lipschitz continuous in $[0,\infty)$.
And hence, $\dd\lim_{t\to\infty}h'(t)=0$ due to $h'(t)\geq 0$ and $\dd\lim_{t\to\infty}h(t)=h_\infty\in(0,\infty)$. Similarly, $\dd\lim_{t\to\infty}g'(t)=0$. \end{proof}

\begin{lem}\lbl{l3.2}\, Let $J$ satisfy the condition {\bf (J)} and $J(x)>0$ in $\mathbb{R}$.
Suppose that $g,\,h\in C^1([0,\infty))$, $g(0)<h(0)$, $g'(t)\le 0$, $h'(t)\ge 0$, and
\eqref{3.2} holds. If $(w,\,g,\,h)$ satisfies, for some positive constants $\beta$ and $M$,
 \[0\le w\le M \ {\rm in }\ D_{g,h}^\infty, \ \ \ w(t,g(t))=w(t,h(t))=0, \ \ \forall\ t\ge 0,\vspace{-2mm}\]
  \bess
 h'(t)\ge\beta\int_{g(t)}^{h(t)}\int_{h(t)}^{\infty}J(x-y)w(t,x)\dy\dx, \ \ \forall\ t>0,
 \eess
and $\dd\lim_{t\to\infty}h'(t)=0$, then
  \[\lim_{t\to\infty}\int_{g(t)}^{h(t)}w(t,x)\dx=0, \ \ \ \int_0^\infty\int_{g(t)}^{h(t)}w(t,x)\dx\dt<\infty.\]
 \end{lem}

\begin{proof}\,
Since $J(x)>0$ in $\mathbb{R}$, we have
 \[\int_{h(t)}^\infty J(x-y)\dy=\int_{-\infty}^{x-h(t)}J(z)\dz\ge \int_{-\infty}^{g_\infty-h_\infty}J(z)\dz=:\sigma_0>0 \ \ \ \forall \ x\in (g(t), h(t)),\; t>0.\]
It follows that
 \[ \frac 1\beta h'(t)\ge\int_{g(t)}^{h(t)}\int_{h(t)}^{\infty}J(x-y)w(t,x)\dy\dx
 \ge\sigma_0 \int_{g(t)}^{h(t)}w(t,x)\dx. \]
Hence
 \bess
 \lim_{t\to\infty}\int_{g(t)}^{h(t)}w(t,x)\dx=0, \ \ \ \int_0^\infty\int_{g(t)}^{h(t)}w(t,x)\dx\dt\leq \frac{h_\infty-h(0)}{\beta\sigma_0}
 \eess
as $\dd\lim_{t\to\infty}h'(t)=0$. The proof is complete. \end{proof}

For $a<b$, $i=1,2$ and $\theta\in C([a, b])$, we define the operator $\mathcal{L}^{d_i}_{(a,b)}+\theta$ by
 $$\left(\mathcal{L}^{d_i}_{(a,b)}+\theta\right)\varphi(x):= d_i\left(\int_a^b J_i(x-y)
 \varphi(y)\dy-\varphi(x)\right)+\theta(x)\varphi(x),\ \ x\in[a, b], \ i=1,2.$$
 The generalized principal eigenvalue of $\mathcal{L}^{d_i}_{(a,b)}+\theta$ is given by
 \[\lambda_p(\mathcal{L}^{d_i}_{(a,b)}+\theta):=\inf\{\lambda\in\mathbb R: (\mathcal{L}^{d_i}_{(a,b)}+\theta) \phi\leq\lambda\phi \mbox{ in $[a,b]$ for some } \phi\in C([a,b]), \;\phi>0\}.\]

\begin{theo}\label{th3.3}\,Assume that $J_1$ and $J_2$ satisfy {\rm\bf (J)}, $J_1(x)>0$, $J_2(x)>0$ in $\mathbb{R}$, and that $(f_1,\,f_2)$ satisfies either \eqref{1.4} or \eqref{1.5}. Let $(u_1, u_2,g,h)$ be the unique solution of \eqref{1.2}. If $h_\infty-g_\infty<\infty$, then
 \bes
  \lim_{t\to\infty}\max_{g(t)\le x\le h(t)}u_i(t,x)=0, \ \ \ i=1,\, 2;
 \lbl{3.4}\ees
moreover,
 \bes
 \lambda_p\left(\mathcal{L}^{d_i}_{(g_\infty, h_\infty)}+a_i\right)\leq 0, \ \
 i=1,\,2.
 \lbl{3.5}\ees
\end{theo}

\begin{proof}\,As $h_\infty-g_\infty<\infty$ and $u_1, u_2$ are bounded, it follows from the
first equation of \eqref{1.2} that $|u_{it}|$ is bounded for $i=1,2$. Lemma \ref{l3.1} then infers that $\dd\lim_{t\to\infty}g'(t)=\lim_{t\to\infty}h'(t)=0$. Since $u_i>0$ and $\mu_i>0$, we have
 \begin{align*}
 h'(t)\ge\mu_i\int_{g(t)}^{h(t)}\int_{h(t)}^{\infty}J_i(x-y)u_i(t,x)\dy\dx, \ \ \forall \ t>0,\ \
 i=1,\, 2.
 \end{align*}
Applying Lemma \ref{l3.2} we thus obtain
 \[\lim_{t\to\infty}\int_{g(t)}^{h(t)}u_i(t,x)\dx=0, \ \; \; \int_0^\infty\int_{g(t)}^{h(t)}u_i(t,x)\dx\dt<\infty, \ \
 i=1,2.\]
We extend $u_i(t,x)$ by 0 for $x\not\in[g(t), h(t)]$ and denote the extended function still by $u_i(t,x)$. Then we may rewrite the above inequality
as
  \[\int_0^\infty\int_{g_\infty}^{h_\infty}u_i(t,x)\dx\dt<\infty.\]
By Fubini's theorem we have
 \[\int_{g_\infty}^{h_\infty}\int_0^\infty u_i(t,x)\dt\dx=\int_0^\infty\int_{g_\infty}^{h_\infty}u_i(t,x)\dx\dt<\infty.\]
Therefore, for $i=1,2$, the function
 \[U_i(x):=\int_0^\infty u_i(t,x)\dt\]
is finite for a.e. $x\in (g_\infty, h_\infty)$.
Since $u_i(t,x)\geq 0$ and $u_{it}\in L^\infty(D^\infty_{g,h})$, it follows, in particular, that
\bes\label{3.6}
\lim_{t\to\infty} u_i(t,x)=0 \ \mbox{ for almost every } x\in [g(0),h(0)].
\ees

Define, for $i=1,2$,
 \[M_i(t):=\max_{x\in [g(t), h(t)]}u_i(t,x).\]
To complete the proof of \eqref{3.4}, it suffices to show that
\bes\label{3.7}
\lim_{t\to\infty} M_i(t)=0 \ \mbox{ for } i=1,2.
\ees

To this end, we need to prove some useful properties of $M_i(t)$ first. Clearly $M_i(t)$ is continuous. Define
  \[ X_i(t):=\{x\in (g(t), h(t)): u_i(t,x)=M_i(t)\}.\]
Then $X_i(t)$ is a compact set for each $t>0$. Therefore, there exist $\underline\xi_i(t),\;\overline\xi_i(t)\in X_i(t)$ such that
  \[u_{it}(t, \underline\xi_i(t))=\min_{x\in X_i(t)}u_{it}(t,x),\quad u_{it}(t, \overline\xi_i(t))=\max_{x\in X_i(t)}u_{it}(t,x).\]
We claim that $M_i(t)$ satisfies, for each $t>0$,
\bes\label{3.8}
\left\{\begin{array}{ll}
M_i'(t+0):=\dd\lim_{s>t, s\to t}\frac{M_i(s)-M_i(t)}{s-t} =u_{it}(t,\overline\xi_i(t)),\\[3mm]
 M_i'(t-0):=\dd\lim_{s<t, s\to t}\frac{M_i(s)-M_i(t)}{s-t} =u_{it}(t,\underline\xi_i(t)).
 \end{array}\right.
\ees
Indeed, for any fixed $t>0$ and $s>t$, we have
 \[ u_i(s,\overline\xi_i(t))-u_i(t,\overline\xi_i(t))\leq M_i(s)-M_i(t)\leq  u_i(s,\overline \xi_i(s))-u_i(t,\overline\xi_i(s)).\]
It follows that
 \bes\label{3.9}
\liminf_{s>t, s\to t}\frac{M_i(s)-M_i(t)}{s-t} \geq u_{it}(t,\overline\xi_i(t)),
 \ees
and
 \[\limsup_{s>t, s\to t}\frac{M_i(s)-M_i(t)}{s-t}\leq \limsup_{s>t, s\to t} \frac{u_i(s,\overline \xi_i(s))-u_i(t,\overline\xi_i(s))}{s-t}.\]
Let $s_n\searrow t$ satisfy
  \[\lim_{n\to\infty}\frac{u_i(s_n,\overline \xi_i(s_n))-u_i(t,\overline\xi_i(s_n))}{s_n-t}= \limsup_{s>t, s\to t} \frac{u_i(s,\overline \xi_i(s))-u_i(t,\overline\xi_i(s))}{s-t}.\]
By passing to a subsequence if necessary, we may assume that $\overline \xi_i(s_n)\to \xi$ as $n\to\infty$. Then $u_i(t, \xi)=\dd\lim_{n\to\infty} M_i(s_n)=M_i(t)$ and hence
$\xi\in X_i(t)$.
Due to the continuity of $u_{it}(t,x)$, it follows immediately that
 \[\lim_{n\to\infty}\frac{u_i(s_n,\overline \xi_i(s_n))-u_i(t,\overline\xi_i(s_n))}{s_n-t}=u_{it}(t,\xi)\leq u_{it}(t,\overline \xi_i(t)).\]
We thus obtain
 \[\limsup_{s>t, s\to t}\frac{M_i(s)-M_i(t)}{s-t}\leq u_{it}(t,\overline \xi_i(t)).\]
Combining this with \eqref{3.9} we obtain
 \[M_i'(t+0)=u_{it}(t,\overline \xi_i(t)).\]
Analogously we can show
 \[M_i'(t-0)=u_{it}(t,\underline \xi_i(t)).\]
Let us note from \eqref{3.8} that $M_i'(t-0)\leq M_i'(t+0)$ for all $t>0$. Therefore if $M_i(t)$ has a local maximum at $t=t_0$, then
$M_i'(t_0)$ exists and $M_i'(t_0)=0$. Moreover, if $M_i(t)$ is monotone nondecreasing for all large $t$ and $\dd\lim_{t\to\infty} M_i(t)=\sigma>0$, then
necessarily $M_i'(t_n-0)\to 0$ along some sequence $t_n\to\infty$; and if $M_i(t)$ is monotone nonincreasing for all large $t$ and $\dd\lim_{t\to\infty} M_i(t)=\sigma>0$, then
necessarily $M_i'(s_n+0)\to 0$ along some sequence $s_n\to\infty$. These properties of $M_i(t)$ will be used below.

We are now ready to prove \eqref{3.7}.
We first consider the situation that $(f_1, f_2)$ satisfies \eqref{1.4}. Arguing indirectly we assume that there exists $i\in\{1,2\}$ such that
the desired identity above does not hold for $M_i(t)$. For definiteness, we assume that $i=1$. Then necessarily
 \[\sigma^*:=\limsup_{t\to\infty} M_1(t)\in (0,\infty).\]
By the above stated properties of $M_i(t)$,  there exists a sequence $t_n>0$ increasing to $\infty$ as $n\to\infty$, and $\xi_n\in\{\underline \xi_1(t_n),\overline\xi_1(t_n)\}$ such that
  \[\lim_{n\to\infty} M_1(t_n)=\sigma^*,\quad \lim_{n\to\infty}u_{1t}(t_n, \xi_n)=0.\]
By passing to a subsequence of $(t_n,\xi_n)$ if necessary, we may assume, without loss of generality,
 \[\lim_{n\to\infty}u_2(t_n, \xi_n)=\rho\in [0,\infty).\]
Since
 \[\lim_{t\to\infty}\int_{g(t)}^{h(t)}u_1(t,x)\dx=0,\]
and $\sup_{x\in\R} J_1(x)<+\infty$, we have
 \[\lim_{t\to\infty}\int_{g(t)}^{h(t)}J_1(x-y)u_1(t,y)\dy=0 \;\mbox{
uniformly for $x\in\R$. }
\]
We now make use of the identity
 \[u_{1t}=d_1\int_{g(t)}^{h(t)}J_1(x-y)u_1(t,y)\dy-d_1u_1+u_1(a_1-b_1u_1-c_1u_2)\]
with $(t,x)=(t_n, \xi_n)$, and take $n\to\infty$ to obtain
 \[0=-d_1 \sigma^*+\sigma^*(a_1-b_1\sigma^*-c_1\rho)<(a_1-d_1)\sigma^*.\]
It follows that $a_1>d_1$. We show next that this leads to a contradiction.

Indeed, by \eqref{3.6}, there exists $x_0\in (g(0), h(0))$ such that
 \[\lim_{t\to\infty} u_i(t,x_0)=0 \ \mbox{ for } i=1,2.\]
Therefore we can find $T>0$ large so that
 \[-d_1+a_1-b_1u_1(t,x_0)-c_1u_2(t,x_0)>(a_1-d_1)/2>0 \ \mbox{ for }t\geq T.\]
It then follows from the equation satisfied by $u_1$ that
 \[u_{1t}(t,x_0)\geq \frac{a_1-d_1}{2}u_1(t,x_0) \ \mbox{ for } t\geq T,\]
which implies $u_1(t,x_0)\to\infty$ as $t\to\infty$, a contradiction to the boundedness of $u_1$.
This completes the proof of \eqref{3.4} for the case that $(f_1, f_2)$ satisfies \eqref{1.4}.

Next we consider the case that \eqref{1.5} is satisfied. The proof mainly follows the above argument for the competition case,
though some small changes are needed. In this case we can similarly show that
$\dd\lim_{t\to\infty} M_1(t)=0$. Hence
  \[\lim_{t\to\infty} \max_{x\in[g(t), h(t)]}u_1(t,x)=0.\]
With this at hand, we may now repeat the argument for the competition case to deduce that $\dd\lim_{t\to\infty} M_2(t)=0$.
We have thus proved \eqref{3.4}.

In the following we prove \eqref{3.5}. Suppose on the contrary that $\lambda_p(\mathcal{L}^{d_1}_{(g_\infty, h_\infty)}+ a_1)>0$. Then $\lambda_p(\mathcal{L}^{d_1}_{(g_\infty+\ep, h_\infty-\ep)}+a_1-\ep)>0$ for small $\ep>0$, say $\ep\in (0,\ep_1)$.
Due to \eqref{3.4}, for such $\ep$, there exists $ T_{\ep}>0$ such that
 \[c_1u_2(t,x)<\ep, \ \ \forall\ t\ge T_{\ep}, \ \ g(t)\le x\le h(t),\]
and
 \[h(t)>h_\infty-\ep, \ \ g(t)<g_\infty+\ep, \ \ \forall\ t\ge T_{\ep}.\]
Consider the auxiliary problem
 \bes\label{3.10}
 \left\{\begin{aligned}
&w_t=d_1\int_{g_\infty+\ep}^{h_\infty-\ep}J_1(x-y)w(t,y)\dy-d_1w +f_\ep(w),
&& t >  T_{\ep},~x\in [g_\infty+\ep, h_\infty-\ep],\\
&w(T_{\ep} ,x)=u_1(T_{\ep}, x), && x\in [g_\infty+\ep, h_\infty-\ep],
 \end{aligned}\right.\ees
where $f_\ep(w)=w(a_1-\ep-b_1w)$. Since $\lambda_p(\mathcal{L}^{d_1}_{(g_\infty+\ep, h_\infty-\ep)}+ f_\ep'(0))>0$, it is well known (see \cite{BZ07, Cov})  that
the  solution $ w_{\ep}(t,x)$ of \eqref{3.10} converges to the unique positive steady state $W_{\ep}(x)$ of \eqref{3.10} uniformly in $[g_\infty+\ep, h_\infty-\ep]$ as $t \to\infty$.
Moreover,  a simple comparison argument yields
 $$u_1(t,x)\ge  w_{\ep}(t,x),~\ \forall~t> T_{\ep},~x\in[g_\infty+\ep,h_\infty-\ep].$$
Thus, there exists $T_{1\ep}> T_{\ep}$ such that
  $$u_1(t,x)\ge {1\over 2} W_{\ep}(x)>0,~\ \forall~t> T_{1\ep}, ~x\in[g_\infty+\ep,h_\infty-\ep].$$
Clearly this is a contradiction to \eqref{3.4}. The proof is complete. \end{proof}

\begin{coro}\label{co3.4}\, Suppose that $J_1$, $J_2$ and $(f_1, f_2)$ satisfy the conditions in Theorem {\rm\ref{th3.3}}, and $(u_1, u_2, g, h)$ is the unique solution of \eqref{1.2}.
If $a_1\geq d_1$ or $a_2\geq d_2$, then necessarily $h_\infty-g_\infty=\infty$.
\end{coro}
\begin{proof}
Arguing indirectly we assume that $h_\infty-g_\infty<\infty$ and $a_i\geq d_i$ for some $i\in\{1,2\}$. Thanks to \cite[Proposition 3.4]{CDLL},
  \[\lambda_p\left(\mathcal{L}^{d_i}_{(g_\infty, h_\infty)}+a_i\right)>0.\]
This is a contradiction to \eqref{3.5}.
\end{proof}

We next consider the case that
 \bes\label{3.11}
a_i<d_i \ \ \mbox{ for }\ i=1,2.
 \ees
In this case, in view of \cite[Proposition 3.4]{CDLL}, $\lambda_p\big(\mathcal{L}^{d_i}_{(0,\ell)}
+a_i\big)<0$ if $0<\ell\ll 1$, and $\lambda_p\big(\mathcal{L}^{d_i}_{(0,\ell)}+a_i\big)>0$
if $\ell\gg 1$, and there exist two positive constants $\ell_1$ and $\ell_2$ such that
  \[\lambda_p\left(\mathcal{L}^{d_i}_{(0,\ell_i)}+a_i\right)=0,\; i=1,2.\]

Define
  \[ \ell_*=\min\{\ell_1,\ell_2\}. \]
We have the following result.

\begin{theo}\label{th3.5}\, Assume that $J_i(x)>0$ in $\R$ for $i=1, 2$,  $(f_1,\,f_2)$ satisfies either \eqref{1.4} or \eqref{1.5}, and \eqref{3.11} holds. Let $(u_1, u_2, g, h)$ be the unique solution of \eqref{1.2}. Then the following conclusions hold:

{\rm (i)} If $h_\infty-g_\infty<\infty$, then $h_\infty-g_\infty\leq \ell_*$.

{\rm (ii)} If $h_0\geq \ell_*/2$, then $h_\infty-g_\infty=\infty$.

{\rm (iii)} If $h_0<\ell_*/2$, then there exist two positive numbers $\Lambda^*\ge\Lambda_*>0$ such that $h_\infty-g_\infty<\infty$ when $0<\mu_1+\mu_2\leq\Lambda_*$ and $h_\infty-g_\infty=\infty$ when  $\mu_1 +\mu_2>\Lambda^*$.
\end{theo}

It is easily seen that conclusions (i) and (ii) in Theorem \ref{th3.5} follow directly from the definition of $\ell_*$ and \eqref{3.5}. We prove (iii) by several lemmas.

\begin{lem}\label{l3.6}\, Under the assumptions of Theorem \ref{th3.5}, there exists a positive number $\Lambda_0$, depending only on $h_0$, $d_i$, $J_i$, $a_i$ and $u_{i0}$, $i=1, 2$, such that $h_\infty-g_\infty<\infty$ for any $\mu_1,\,\mu_2$ satisfying $0<\mu_1+\mu_2 <\Lambda_0$.
\end{lem}

We need some comparison results to prove this lemma. The proof of the following Lemmas \ref{l3.7} and \ref{l3.8} can be carried out by a combination of the proofs of
\cite[Theorem 3.1]{CDLL}, \cite[Lemma 5.1]{GW} and \cite[Lemma 4.1]{Wjde14}. Since the adaptation is rather straightforward, we omit the details here.

\begin{lem}\label{l3.7}\, For $T\in(0,\infty)$,
suppose that $\bar h,\,\bar g\in C([0,T])$, $\bar u_1,\,\bar u_2\in  C\left(\{0\le t\le T,\,
\bar g(t)\le x\le\bar h(t)\}\right)$ and satisfy
 \bes\label{3.12}
 \hspace{1.2cm} \left\{\begin{aligned}
&\bar u_{1t}\dd\ge d_1\int_{\bar g(t)}^{\bar h(t)}J_1(x-y)\bar u_1(t,y)\dy
-d_1\bar u_1+\bar u_1(a_1-b_1\bar u_1), &&0<t\le T,~\bar g(t)<x<\bar h(t),\\
&\bar u_{2t}\dd\ge d_2\int_{\bar g(t)}^{\bar h(t)}J_2(x-y)\bar u_2(t,y)\dy-d_2\bar u_2
+\bar u_2(a_2-b_2\bar u_2), &&0<t\le T,~\bar g(t)<x<\bar h(t),\\
 &\bar u_i(t,\bar g(t))\ge 0,\ \
 \bar u_i(t,\bar h(t))\geq 0,\ \ i=1,\,2, &&0<t\le T,\\
 &\dd\bar h'(t)\ge\sum_{i=1}^2\mu_i\int_{\bar g(t)}^{\bar h(t)}\!\int_{\bar h(t)}^{\infty}\!
 J_i(x-y)\bar u_i(t,x)\dy\dx,&& 0\le t\le T,\\
 &\dd\bar g'(t)\le-\sum_{i=1}^2\mu_i\int_{\bar g(t)}^{\bar h(t)}\!\int_{-\infty}^{\bar g(t)}\!
 J_i(x-y)\bar u_i(t,x)\dy\dx, && 0\le t\le T,\\
 &\bar u_i(0,x)\ge u_{i0}(x),\ \ i=1,\,2,~\ \bar h(0)\ge h_0,~\ \bar g(0)\le-h_0, \ && |x|\le h_0.
 \end{aligned}\right.\ees
Let $(u_1, u_2, g, h)$ be the unique solution of \eqref{1.2} with $(f_1, f_2)$ satisfying
\eqref{1.4}. Then
  \[u_1\leq\bar u_1,~ \ u_2\leq\bar u_2,~ \ g\geq\bar g,~\ h\leq\bar h~\ \ \mbox{in} \ \ D^T_{g, h}.\]
\end{lem}

\begin{lem}\label{l3.8}\, In Lemma $\ref{l3.7}$, if we replace
the second inequality in \eqref{3.12} by
 \[\bar u_{2t}\dd\ge d_2\int_{\bar g(t)}^{\bar h(t)}J_2(x-y)\bar u_2(t,y)\dy-d_2\bar u_2
+\bar u_2(a_2-b_2\bar u_2+c_2\bar u_1), \ \ 0<t\le T,~\bar g(t)<x<\bar h(t),\]
and let $(u_1, u_2, g, h)$ be the unique solution of \eqref{1.2} with $(f_1, f_2)$ satisfying
\eqref{1.5}, then the conclusion still holds true.
\end{lem}

\begin{proof}[Proof of Lemma \ref{l3.6}]\, The idea of this proof comes from \cite[Theorem 3.12]{CDLL}, \cite[Lemma 5.2]{Wjde14} and \cite[Lemma 4.4]{WZhang16}.

Since $\lambda_p\big(\mathcal{L}^{d_i}_{(-h_0,h_0)}+a_i\big)<0$, $i=1,2$, we can choose $h_1>h_0$
such that
 \[\lambda_i:=\lambda_p\big(\mathcal{L}^{d_i}_{(-h_1,h_1)}+a_i\big)<0, \ \ i=1,2.\]

{\it Case 1: The competition case}. Suppose that $(f_1, f_2)$ satisfies \eqref{1.4}.
Let $w_i(t,x)$ be the unique solution of
 \begin{equation}\label{3.13}
 \left\{\begin{aligned}
&w_{it}=d_i\int_{-h_1}^{h_1}J_i(x-y)w_i(t,y)\dy-d_iw_i+a_iw_i,
  & &t>0,~|x|\le h_1,\\
 &w_i(0,x)=u_{i0}(x), & &  |x|\le h_0, \\
 & w_i(0,x)=0, & &  |x|>h_0.
 \end{aligned}\right.\end{equation}
Let $\varphi_i>0$ be the corresponding
normalized eigenfunction of $\lambda_i$, namely $\|\varphi_i\|_\infty=1$ and
 \[\left(\mathcal{L}^{d_i}_{(-h_1, h_1)}+a_i\right)[\,\varphi_i](x)=\lambda_i\varphi_i(x), \ \ \forall\ |x|\le h_1.\]
For $C>0$ and $z_i(t,x)=C e^{\lambda_i t/2}\varphi_i(x)$, it is easy to check that
 \begin{align*}
& d_i\int_{-h_1}^{h_1}J_i(x-y)z_i(t,y)\dy-d_i z_i+a_i z_i-z_{it}\\
 =&\; Ce^{\lambda_i t/2 }\left(d_i\int_{-h_1}^{h_1}J_i(x-y)\varphi_i(y)\dy-d_i\varphi_i+a_i\varphi_i
 -\frac{\lambda_i} 2\varphi_i\right)\\
 =&\;\frac{\lambda_i} 2C e^{\lambda_i t/2}\varphi_i(x)<0, \ \ \forall~t>0, ~|x|\le h_1, \ \ i=1,\,2.
  \end{align*}
Choose $C>0$ large such that $C\varphi_i>u_{i0}$ on $[-h_1, h_1]$. Then we can apply
\cite[Lemma 3.3]{CDLL} to $w_i-z_i$ to conclude that
 \begin{equation}\label{3.14}
 w_i(t,x)\le z_i(t,x)=C e^{\lambda_i t/2}\varphi_i(x)\leq C e^{\lambda_i t/2}, \ \
  \forall~t>0, ~|x|\le h_1.
\end{equation}

Set
 \[\lambda=\max\{\lambda_1,\,\lambda_2\}, \ \
 r(t)=h_0+2(\mu_1+\mu_2)Ch_1\int_0^t e^{\lambda s/2}{\rm d}s, \ \ \eta(t)=-r(t), \ \ t\ge0.\]
Then $\lambda<0$.

We claim that $(w_1, w_2, \eta, r)$ is an upper solution of (\ref{1.2})
with $(f_1, f_2)$ satisfying \eqref{1.4}. Firstly, we compute for $t>0$,
 $$r(t)=h_0-2(\mu_1+\mu_2)Ch_1\frac{2}{\lambda}\left(1-e^{\lambda t/2}\right)<h_0-2(\mu_1 +\mu_2)Ch_1\frac{2}{\lambda}\leq h_1$$
provided that
  $$0<\mu_1+\mu_2\leq \Lambda_0:=\frac{-\lambda(h_1-h_0)}{4Ch_1}.$$
Similarly, for such $\mu_1$ and $\mu_2$, we have $\eta(t)>-h_1 $ for any $t>0$. Thus \eqref{3.13} gives
 \[w_{it}\ge d_i\int_{\eta(t)}^{r(t)}J_i(x-y)w_i(t,y)\dy-d_iw_i+w_i(a_i-b_iw_i),
 \ \ t>0,~x\in[\eta(t),\,r(t)].\]
Secondly, due to \eqref{3.14}, it is easy to check that
 \[\int_{\eta(t)}^{r(t)}\int_{r(t)}^{\infty}J_i(x-y) w_i(t,x)\dy\dx
  \le 2Ch_1e^{\lambda_it/2}\le 2Ch_1e^{\lambda t/2}.\]
Thus
 $$r'(t)=2(\mu_1+\mu_2)Ch_1e^{\lambda t/2}\ge
 \sum_{i=1}^2\mu_i\int_{\eta(t)}^{r(t)}\int_{r(t)}^{\infty}J_i(x-y)w_i(t,x)\dy\dx.$$
Similarly, one has
  \bes
  \eta'(t)\le-\sum_{i=1}^2\mu_i\int_{\eta(t)}^{r(t)}\int^{\eta(t)}_{-\infty}
  J_i(x-y)w_i(t,x)\dy\dx.
  \lbl{3.15}\ees

The above arguments show that $(w_1, w_2, \eta, r)$ is an upper solution of (\ref{1.2}) with $(f_1, f_2)$ satisfying \eqref{1.4}. By Lemma \ref{l3.7} we get
 \[u_1(t,x)\leq w_1(t,x),~ \ u_2(t,x)\leq w_2(t,x),~ \ g(t)\geq\eta(t),~\ h(t)\leq r(t)\]
for all $t\geq 0$, $g(t)\le x\le h(t)$. Therefore
  $$h_\infty-g_\infty\le\lim_{t\to\infty}\left[r(t)-\eta(t)\right]
  \leq 2h_1<\infty.$$

{\it Case 2: The prey-predator case}. Suppose that $(f_1, f_2)$ satisfies \eqref{1.5}.
Inspired by \cite[Lemma 5.2]{Wjde14}, let $w_1(t,x)$ and $w_2(t,x)$ be the unique solution of
 \bess
 \left\{\begin{aligned}
&w_{1t}=d_1\int_{-h_1}^{h_1}J_1(x-y)w_1(t,y)\dy-d_1w_1+a_1w_1,
  & &t>0,~|x|\le h_1,\\
 &w_1(0,x)=u_{10}(x), & &  |x|\le h_0, \\
 & w_1(0,x)=0, & &  |x|>h_0
 \end{aligned}\right.\eess
and
 \bess
 \left\{\begin{aligned}
&w_{2t}=d_2\int_{-h_1}^{h_1}J_2(x-y)w_2(t,y)\dy-d_2w_2+w_2(a_2-b_2w_2+c_2w_1),
  & &t>0,~|x|\le h_1,\\
 &w_2(0,x)=u_{20}(x), & &  |x|\le h_0, \\
 & w_2(0,x)=0, & &  |x|>h_0,
 \end{aligned}\right.\eess
respectively. Take $\lambda_i$ and $\varphi_i$ as above. Then there exists $0<\sigma\le 1$ such that
  \bes
   \sigma c_2\varphi_1(x)\le b_2\varphi_2(x), \ \ \ \forall \ |x|\le h_1.
  \lbl{3.16}\ees
Choose $C>0$ large such that
 \[\sigma C\varphi_1(x)>u_{10}(x), \ \ \ C\varphi_2(x)>u_{20}(x), \ \ \ \forall \ |x|\le h_1.\]
Set $z_1(t,x)=\sigma C e^{\lambda_1 t/2}\varphi_1(x)$, $\lambda=\max\{\lambda_1,\,\lambda_2\}$ and
$z_2(t,x)=C e^{\lambda t/2}\varphi_2(x)$. Similar to the above,
 \begin{align*}
&\; d_1\int_{-h_1}^{h_1}J_1(x-y)z_1(t,y)\dy-d_1 z_1+a_1z_1-z_{1t}\\
 =&\;\sigma Ce^{\lambda_1 t/2 }\left(d_1\int_{-h_1}^{h_1}J_1(x-y)\varphi_1(y)\dy-d_1\varphi_1+a_1\varphi_1
 -\frac{\lambda_1} 2\varphi_1\right)\\
 =&\;\frac{\lambda_1} 2\sigma C e^{\lambda_1 t/2}\varphi_1(x)<0.
  \end{align*}

Now we consider $z_2(t,x)$. Using $\lambda=\max\{\lambda_1,\,\lambda_2\}<0$
and \eqref{3.16}, we obtain
\begin{align*}
& d_2\int_{-h_1}^{h_1}J_2(x-y)z_2(t,y)\dy-d_2 z_2+z_2(a_2-b_2z_2+c_2z_1)-z_{2t}\\
 =&\;Ce^{\lambda t/2 }\left(d_2\int_{-h_1}^{h_1}J_2(x-y)\varphi_2(y)\dy-d_2\varphi_2+a_2\varphi_2
 -\frac{\lambda} 2\varphi_2\right)\\
 &\;+C^2 e^{\lambda t/2}\varphi_2\left(\sigma c_2e^{\lambda_1 t/2}\varphi_1-b_2
 e^{\lambda t/2}\varphi_2\right)\\
 =&\;Ce^{\lambda t/2}\left(\lm_2\varphi_2-\frac{\lambda} 2\varphi_2\right)+C^2 e^{\lambda t}\varphi_2\left(\sigma c_2e^{(\lambda_1-\lm)t/2}\varphi_1-b_2\varphi_2\right)\\
 \le&\;\frac{\lambda} 2C e^{\lambda t/2}\varphi_2(x)<0.
  \end{align*}
Similar to the above, applying \cite[Lemma 3.3]{CDLL} to $w_i-z_i$ we have
 \bes\left\{\begin{aligned}\label{3.17}
 &w_1(t,x)\le z_1(t,x)=\sigma C e^{\lambda_1 t/2}\varphi_1(x)\leq \sigma C e^{\lambda t/2}, \ \
  &&\forall~t>0, ~|x|\le h_1,\\
 &w_2(t,x)\le z_2(t,x)=C e^{\lambda t/2}\varphi_2(x)\leq C e^{\lambda t/2}, \ \
  &&\forall~t>0, ~|x|\le h_1.
  \end{aligned}\right.\ees
Set
 \[ r(t)=h_0+2(\sigma\mu_1+\mu_2)Ch_1\int_0^t e^{\lambda s/2}{\rm d}s, \ \ \eta(t)=-r(t), \ \ t\ge0. \]
 For $t>0$, we have
 $$r(t)=h_0-2(\sigma\mu_1+\mu_2)Ch_1\frac{2}{\lambda}\left(1-e^{\lambda t/2}\right)
 <h_0-2(\sigma\mu_1+\mu_2)Ch_1\frac{2}{\lambda}\leq h_1$$
provided that
  $$0<\sigma\mu_1+\mu_2\leq \Lambda_0:=\frac{-\lambda(h_1-h_0)}{4Ch_1}.$$
Similarly, for such $\mu_1$ and $\mu_2$, $\eta(t)>-h_1 $ for any $t>0$. Thus we have
\begin{align*}
 w_{1t}\ge&\;\dd d_1\int_{\eta(t)}^{r(t)}J_1(x-y)w_1(t,y)\dy-d_1w_1+a_1w_1,\\
 w_{2t}\ge&\;d_2\dd\int_{-h_1}^{h_1}J_2(x-y)w_2(t,y)\dy-d_2w_2+w_2(a_2-b_2w_2+c_2w_1)
 \end{align*}
for $t>0$ and $x\in[\eta(t),\,r(t)]$. On the other hand, due to \eqref{3.17}, it is easy to check that
 \begin{align*}
 \int_{\eta(t)}^{r(t)}\int_{r(t)}^{\infty}J_1(x-y)w_1(t,x)\dy\dx
  \le&\;2\sigma Ch_1e^{\lambda t/2}, \\[1mm]
  \int_{\eta(t)}^{r(t)}\int_{r(t)}^{\infty}J_2(x-y)w_2(t,x)\dy\dx
  \le&\; 2 Ch_1e^{\lambda t/2}.\end{align*}
Thus
 $$r'(t)=2(\sigma\mu_1+\mu_2)Ch_1e^{\lambda t/2}\ge
 \sum_{i=1}^2\mu_i\int_{\eta(t)}^{r(t)}\int_{r(t)}^{\infty}J_i(x-y)w_i(t,x)
 \dy\dx.$$
Similarly, $\eta(t)$ satisfies \eqref{3.15}. We may now apply Lemma \ref{l3.8} to conclude that, when $\sigma\mu_1+\mu_2\leq \Lambda_0$,
  $$h_\infty-g_\infty\le\lim_{t\to\infty}\left[r(t)-\eta(t)\right]
  \leq 2h_1<\infty.$$
As $\mu_1+\mu_2\leq \Lambda_0$ implies $\sigma\mu_1+\mu_2\leq \Lambda_0$, the desired result
is proved.
\end{proof}

To complete the proof of Theorem \ref{th3.5}, it remains to show that if $\mu_1+\mu_2$ is large, then $h_\infty-g_\infty=\infty$. We need the following lemma.

\begin{lem}\lbl{l3.9}\, Let {\bf (J)} hold for the kernel function $J$, and $C>0$ be a constant.
For any given constants $s_0, H>0$, and any function $w_0\in C([0,s_0])$
satisfying $w_0(\pm s_0)=0$ and $w_0>0$ in $(-s_0,s_0)$,
there exists $\mu^0>0$, depending on $J(x)$, $d$, $C$, $w_0(x)$ and $s_0$,
such that if $\mu \geq\mu^0$ and $(w, s, c)$ satisfies
 \bess\left\{\begin{aligned}
 &w_t\ge d\dd\int_{c(t)}^{s(t)}J(x-y)w(t,y)\dy-d w-Cw, & &t>0,~c(t)<x<s(t),\\
 &w(t,c(t))=w(t,s(t))=0, & &t>0,\\
 &s'(t)\ge\dd\mu \int_{c(t)}^{s(t)}\int_{s(t)}^{\infty}J(x-y)w(t,x)\dy\dx,\ & &t>0,\\
 &c'(t)\le-\dd\mu \int_{c(t)}^{s(t)}\int_{-\infty}^{c(t)}J(x-y)w(t,x)\dy\dx, & &t>0,\\
 &w(0,x)=w_0(x),~s(0)=-c(0)=s_0, & &|x|\le s_0,
 \end{aligned}\right.
 \eess
then $\dd\liminf_{t\to\infty}[s(t)-c(t)]>H$.
\end{lem}

\begin{proof}\, We adapt the approach of  \cite[Lemma 3.2]{WZjdde14}.

Firstly, the comparison principle gives
 \[ w(t,x)>0, \ \ \ \forall \ t>0, \ c(t)<x<s(t).\]
 It then follows that $s'(t)>0, \ \ c'(t)<0$ for $t>0$.

Take a function $b(t)\in C^1([0,1])$ satisfying $b(t)>0$ in $[0,1]$,
$b(0)=s_0$ and $b(1)=H$, and set $a(t)=-b(t)$. Consider the following problem
 \bess\left\{\begin{aligned}
&z_t=d\int_{a(t)}^{b(t)}J(x-y)z(t,y)\dy-dz-Cz, & &0<t<1,~a(t)<x<b(t),\\
&z(t,b(t))=z(t,a(t))=0,& &0<t<1,\\
&z(0,x)=w_0(x),& &|x|\le s_0.
\end{aligned}\right.
 \eess
In view of \cite[Lemma 2.3]{CDLL}, this problem has a unique solution $z$
which is continuous on $\{0\le t\le 1,\,a(t)\le x\le b(t)\}$ and satisfies
$z(t,x)>0$ for all $t\ge 0$ and $a(t)<x<b(t)$. Thus the functions
 \begin{align*}
 r(t):=\int_{a(t)}^{b(t)}\int_{b(t)}^{\infty}J(x-y)z(t,x)\dy\dx,\ \ \
 l(t):=\int_{a(t)}^{b(t)}\int_{-\infty}^{a(t)}J(x-y)z(t,x)\dy\dx\end{align*}
are positive and continuous on $[0, 1]$, and so $r(t), l(t)\ge\sigma>0$
on $[0, 1]$ for some constant $\sigma$. Since $a'(t)$ and $b'(t)$ are bounded
on $[0, 1]$, we can find $\mu^0>0$ such that when $\mu \geq\mu^0$, there hold:
 \begin{align*}
 b'(t)&\;\le\mu r(t)=\mu \int_{a(t)}^{b(t)}\int_{b(t)}^{\infty}J(x-y)z(t,x)\dy\dx, \\
 a'(t)&\;\ge-\mu l(t)=-\mu\int_{a(t)}^{b(t)}\int_{-\infty}^{a(t)}J(x-y)z(t,x)\dy\dx
 \end{align*}
for all $0\le t\le 1$. Applying the comparison principle we get
 \[c(t)\le a(t), \ \ s(t)\ge b(t), \ \ \forall\ 0\le t\le 1,\]
and so $s(1)-c(1)\ge b(1)-a(1)=2H$ when $\mu \geq\mu^0$. The desired conclusion now follows directly
and the proof if complete.
\end{proof}

\begin{proof}[Completiton of the proof of Theorem {\rm\ref{th3.5}}]  Since $u_1$ and $u_2$ are bounded, there exists $C>0$ such that
 \[a_1-b_1u_1(t,x)-c_1u_2(t,x)>-C.\]
We thus have
 \bess\left\{\begin{aligned}
 &u_{1t}\ge d_1\dd\int_{g(t)}^{h(t)}J_1(x-y)u_1(t,y)\dy-d_1u_1-Cu_1,
 && t>0,~g(t)<x<h(t),\\
 &u_1(t,g(t))=u_1(t,h(t))=0, &&t\ge 0,\\
 &h'(t)\ge\dd\mu_1\int_{g(t)}^{h(t)}\int_{h(t)}^{\infty}J_1(x-y)u_1(t,x)\dy\dx, &&t\ge 0,\\
 &g'(t)\le-\dd\mu_1\int_{g(t)}^{h(t)}\int_{-\infty}^{g(t)}J_1(x-y)u_1(t,x)\dy\dx,\ \ &&t\ge 0,\\
 &u_1(0,x)=u_{10}(x),\ \  &&|x|\le h_0,\\
 &h(0)=-g(0)=h_0.
 \end{aligned}\right. \eess
In view of Lemma \ref{l3.9}, there exists $\mu_1^0>0$ such that $h_\infty-g_\infty>\ell_*$ when $\mu_1>\mu_1^0$. Similarly, $h_\infty-g_\infty>\ell_*$ when $\mu_2>\mu_2^0$ for some $\mu_2^0>0$.
Take $\Lambda^0=\mu_1^0+\mu_2^0$. Then $h_\infty-g_\infty>\ell_*$ when $\mu_1+\mu_2>\Lambda^0$, and hence
$h_\infty-g_\infty=\infty$ by the conclusion (i).
\end{proof}

Clearly Theorems 1.2 and 1.3 follow directly from Theorem \ref{th3.3}, Corollary 3.4 and Theorem \ref{th3.5}.

\subsection{Long-time behaviour in the case of spreading}

In this subsection, we examine the long-time behaviour of the solution to \eqref{1.2} when $h_\infty-g_\infty=\infty$. For simplicity, we only consider two situations, namely the
  weak competition case and the weak predation case as described in Theorem \ref{th1.4}.

\begin{prop}\label{p3.10}\, Let $(u_1,u_2,g,h)$ be the unique solution of \eqref{1.2} with $h_\infty-g_\infty=\infty$. Then $g_\infty=-\infty$ and $h_\infty=\infty$ in the weak competition case and in the weak predation case.
\end{prop}

\begin{proof}\, We first consider the weak competition case. By a simple comparison argument involving the ODE problem $v'=v(a_1-b_1 v),\; v(0)=\|u_{10}\|_\infty$, we easily see that, for any small $\varepsilon>0$, there exists $T=T_\varepsilon>0$ large such that
  \[ u_1(t,x)\leq {a_1}/{b_1}+\varepsilon \ \mbox{ for } t\geq T,\; x\in [g(t), h(t)].\]
Since $b_1/c_2>a_1/a_2$, we may assume that $\tilde a_2:=a_2-c_2({a_1}/{b_1}+\varepsilon)>0$.
Thus $(u_2,g,h)$ satisfies
 \bes\label{3.18}
 \hspace{1cm}\left\{\begin{aligned}
 &u_{2t}\ge d_2\dd\int_{g(t)}^{h(t)}J_2(x-y)u_2(t,y)\dy-d_2u_2+u_2(\tilde a_2-b_2u_2), && t\geq T,~g(t)<x<h(t),\\
&u_2(t,g(t))=u_2(t,h(t))=0, &&t\ge T,\\
&h'(t)\geq\dd\mu_2\int_{g(t)}^{h(t)}\int_{h(t)}^{\infty}J_2(x-y)u_2(t,x)\dy\dx, &&t\ge T,\\
 &g'(t)\le-\dd\mu_2\int_{g(t)}^{h(t)}\int_{-\infty}^{g(t)}
  J_2(x-y)u_2(t,x)\dy\dx, &&t\ge T.\\
\end{aligned}\right.
 \ees
 Consider the following auxiliary problem
 \bes\label{3.19}
 \left\{\begin{aligned}
 &w_t=d_2\dd\int_{\eta(t)}^{r(t)}J_2(x-y)w(t,y)\dy-d_2w+w(\tilde a_2-b_2w), && t>T,~\eta(t)<x<r(t),\\
 &w(t,\eta(t))=w(t,r(t))=0, &&t\ge T,\\
 &r'(t)=\dd\mu_2\int_{\eta(t)}^{r(t)}\int_{r(t)}^{\infty}J_2(x-y)w(t,x)\dy\dx, &&t\ge T,\\
 &\eta'(t)=-\dd\mu_2\int_{\eta(t)}^{r(t)}\int_{-\infty}^{\eta(t)}
 J_2(x-y)w(t,x)\dy\dx, &&t\ge T,\\
 &w(T,x)=u_2(T,x),\ \ &&g(T)\leq x\leq h(T),\\
 &\eta(T)=g(T), \ \ r(T)=h(T).
\end{aligned}\right.
 \ees
 By the comparison principle, the unique solution $(w,\eta, r)$ of \eqref{3.19} satisfies
 \[ w(t,x)\leq u_2(t,x), \ \ \; \eta(t)\geq g(t), \ \ \; r(t)\leq h(t). \]

 In view of \cite[Theorem 1.3]{CDLL}, if $d_2\leq \tilde a_2$, then spreading happens to \eqref{3.19} and hence $\eta(t)\to-\infty$,
 $r(t)\to\infty$ as $t\to\infty$, which imply that $g_\infty=-\infty, \; h_\infty=\infty$.

If $d_2> \tilde a_2$, then \cite[Theorem 1.3]{CDLL} infers the existence of a unique $\ell_2$ such that spreading happens to \eqref{3.19} provided $h(T)-g(T)\geq \ell_2$. The latter is guaranteed to happen if $T$ is large enough, since $h_\infty-g_\infty=\infty$. Therefore in either case, we must have $g_\infty=-\infty,\; h_\infty=\infty$.

We now consider the predator-prey case. This time \eqref{3.18} holds for any $T>0$ with $\tilde a_2$ replaced by $a_2$. Hence the same argument shows that $g_\infty=-\infty,\; h_\infty=\infty$.
The proof is complete. \end{proof}

\begin{rem}\label{r3.11}\, The assumptions in Proposition {\rm\ref{p3.10}} can be relaxed.
From the above proof, it is easily seen that the conclusion holds whenever $(f_1, f_2)$ satisfies \eqref{1.5}, and in the competition case \eqref{1.4}, the conclusion holds if either ${b_1}/{c_2}>{a_1}/{a_2}$ or ${a_1}/{a_2}>{c_1}/{b_2}$.
\end{rem}

The remaining part of this paper is devoted to the proof of Theorem \ref{th1.4}.
We start with several preparitory results. Consider the auxiliary problem
\begin{equation}\label{3.20}
\begin{cases}
u_t=d\displaystyle\int_{\R}J(x-y)u(t,y)dy-du
+u(a(x)-bu),&t>0,\ x\in\R,\\
u(0,x)=u_0(x),&x\in\R,
\end{cases}
\end{equation}
where $a\in C(\R)\cap L^\infty(\R)$, $d$ and $b$ are  positive constants, and $J$ satisfies {\bf (J)}.

\begin{prop}\label{p3.12}
Let $a$, $b$, $d$ and $J$ be as given above. Then the following conclusions hold:

{\rm (i)} For any bounded interval $\Omega$, the principal eigenvalue $\lambda_p(\mathcal{L}_\Omega^d+a)$ is strictly increasing in $a$.

{\rm (ii)} If $\lambda_p(\mathcal{L}_{\R}^d+a):=\dd\lim_{l\to\infty}
\lambda_p(\mathcal{L}_{(-l,l)}^d+a)>0$, then  problem \eqref{3.20} admits a unique positive steady state $U(x)$. Moreover, for any non-negative initial function $u_0\in C(\R)\cap L^\infty(\R)$, $u_0\not\equiv 0$, the unique solution of \eqref{3.20} satisfies
 \[\lim_{t\to\infty}u(t,x)=U(x) \ \text{\ locally uniformly in\ } \R.\]
\end{prop}

\begin{proof}
Conclusion (i) follows from \cite[Proposition 1.1 (ii)]{Cov}. Conclusion (ii) can be obtained by similar arguments as in \cite[Section 4]{B-JMB16}.
\end{proof}

Next we consider
the auxiliary problem
\begin{equation}\label{3.21}
\begin{cases}
u_t=d\displaystyle\int_{-l}^lJ(x-y)u(t,y)dy-du
+u(a_l-bu),&t>0,\ x\in[-l,l],\\
u(0,x)=u_0(x),&x\in[-l,l],
\end{cases}
\end{equation}
where $a_l,\ d,\ b,\; l$  are positive constants, $J$ satisfies {\bf (J)}, and $a_l$ satisfies
 \[\lim_{l\to\infty}a_l=a>0.\]

\begin{lem}\label{l3.13}
Under the above assumptions,  there exists $L>0$ large such that  for any $l>L$ and any $u_0\in C([-l,l])$ satisfying $u_0\geq,\not\equiv0$, the following conclusions hold:

{\rm (i)} \ The unique solution of \eqref{3.21} satisfies
 \[\lim_{t\to\infty}u_l(t,x)=u_l(x) \
\text{\ uniformly for \ } x\in[-l,l],\]
where $u_l(x)$ is the unique positive solution of
 \[d\displaystyle\int_{-l}^lJ(x-y)u(y)dy-du+u(a_l-bu)=0,\ \ x\in[-l,l];\]

{\rm (ii)}\
 \[\lim_{l\to\infty}u_l(x)={a}/b \
\text{\ locally uniformly in\ } x\in\R.\]
\end{lem}

\begin{proof}
Given any $\varepsilon>0$ small, we can find $L_0>0$ such that
  \[0<a - \varepsilon<a_l<a+\varepsilon \ \mbox{ for } l>L_0.\]
It follows that, for such $l$,
  \[\lambda_p\left(\mathcal{L}_{(-l,l)}^d+a_l\right)>
  \lambda_p\left(\mathcal{L}_{(-l,l)}^d+a-\varepsilon\right)\to a-\varepsilon>0 \mbox{ as } l\to\infty.\]
Therefore
there exists $L\geq L_0$ such that
  \[\lambda_p\left(\mathcal{L}_{(-l,l)}^d+a_l\right)>0 \ \text{\ for\ } l>L,\]
and by  \cite{BZ07, Cov} and Proposition 3.4 of \cite{CDLL} we can conclude that for $l>L$, (i) holds.

To show (ii), we note that for $l>L$, $u_l(t,x)$ is a super-solution to \eqref{3.21} with $a_l$ replaced by $a-\varepsilon$, whose unique solution we denote by $u_{l,\epsilon}(t,x)$. Hence $u_l(t,x)\geq u_{l,\varepsilon}(t,x)$. By \cite{BZ07, Cov} again, we see that as $t\to\infty$, $u_{l,\varepsilon}(t,x)\to u_{l,\varepsilon}(x)$ uniformly in $[-l,l]$, with $u_{l,\varepsilon}(x)$ the unique steady state of the problem. We thus obtain $u_l(x)\geq u_{l,\varepsilon}(x)$. According to \cite[Proposition 3.6]{CDLL}, we have
 \[\lim_{l\to\infty} u_{l,\varepsilon}(x)=(a-\varepsilon)/b \ \mbox{ locally uniformly in } \mathbb R.\]
It follows that $\dd\liminf_{l\to\infty}u_l(x)\geq (a-\varepsilon)/b $ locally uniformly in $\mathbb R$. The arbitrariness of $\varepsilon$ then infers that
 \[\liminf_{l\to\infty}u_l(x)\geq a/b \ \mbox{ locally uniformly in } \mathbb R.\]

Analogously we can show
 \[\limsup_{l\to\infty}u_l(x)\leq a/b \ \mbox{ locally uniformly in } \mathbb R.\]
Therefore (ii) holds.
\end{proof}

When spreading happens, in the local diffusion case, to study the long-time behavior  of diffusive population systems with free boundaries, a key tool is an  iteration method, which  has been widely used in, for example, \cite{Wjde14, WZjdde14, WZjdde17, ZW2015}. To adapt this method to the nonlocal diffusion case here, we rely on
the following technical lemma.

\begin{lem}\label{l3.14}\, { Let $g(t)<h(t)$ be two continuous functions satisfying
 \[\lim_{t\to\infty}g(t)=-\infty,\; \ \ \lim_{t\to\infty} h(t)=\infty.\]
Let $K_0$ be a positive constant, $w$ be a continuous function satisfying $|w(t,x)|\leq K_0$ for $t>0,\; x\in [g(t), h(t)]$. Suppose that $u$ satisfies
\begin{equation*}
\begin{cases}
u_t=d\displaystyle\int_{g(t)}^{h(t)}J(x-y)u(t,y)dy-du
+u(a-bu-w(t,x)),&t>0,\ x\in(g(t),h(t)),\\
u(t,g(t))=u(t,h(t))=0,&t>0,\\
u(0,x)=u_0(x),\ -g(0)=h(0)=h_0,&x\in(-h_0,h_0),
\end{cases}
\end{equation*}
where $a,b,d,h_0$ are positive constants, $J$ satisfies {\rm \bf (J)}, $u_0\in C([-h_0,h_0])$ is nonnegative and not identically $0$. }
Then the following statements hold:

{\rm (i)} If for some constant $m\in[-K_0, K_0]$,
\begin{equation}\label{3.22}
\liminf_{t\to\infty}w(t,x)\geq m
\ \text{\ locally uniformly in\ } \R,
\end{equation}
then
\[\limsup_{t\to\infty}u(t,x)\leq[a- m]_+/b
\ \text{\ locally uniformly in\ } \R,\]
where $[\ \cdot\ ]_+$ is defined by $[\theta]_+=\max\{\theta, 0\}$.

{\rm (ii)} If $a>M$ and
\begin{equation}\label{3.23}
\limsup_{t\to\infty}w(t,x)\leq M
\ \text{\ locally uniformly in\ } \R
\end{equation}
 for some constant $M$, then
\[\liminf_{t\to\infty}u(t,x)\geq(a-M)/b
\ \text{\ locally uniformly in\ } \R.\]
\end{lem}

\begin{proof}
(i) For any integer $n\geq 1$, it follows from \eqref{3.22} that there exists $T_n$ such that
  \[w(t,x)\geq m-1/n \ \text{\ for\ }t\geq T_n \ \text{\ and\ } x\in[-n-1,n+1].\]
For any given small $\varepsilon>0$, define
  \begin{equation*}
\sigma_n=\begin{cases}
a-m+1/n,&a-m>0,\\
\varepsilon+1/n,&a-m\leq 0,
\end{cases}
\end{equation*}
and
 \begin{align*}
&a_n(x)=\begin{cases}
\sigma_n,&|x|<n,\\
\sigma_n+2(a+K_0+1-\sigma_n)(|x|-n),&n\leq|x|\leq n+1/2,\\
a+K_0+1,&|x|>n+1/2.
\end{cases}
\end{align*}
Clearly $a_n\in C(\R)$, $a-w(t,x)\leq a_n(x)$ for $t>T_n$ and $x\in\R$, $a_n(x)$ is nonincreasing in $n$  and
\begin{equation*}
\lim_{n\to\infty}a_n(x)=\sigma_\infty:=
\begin{cases}
a-m,&a-m>0,\\
\varepsilon,&a-m\leq0.
\end{cases}
\end{equation*}

Let $K:=\max\left\{(a+K_0)/b,\, \|u_0\|_\infty\right\}$. It follows from the comparison principle (\cite[Lemma 2.2]{CDLL}) that
  \[u(t,x)\leq K \ \ \text{\ for\ }\ t\geq0, \ \ x\in[g(t),h(t)].\]

Let $z_n$ be the unique solution of
\begin{equation}\label{3.24}
\begin{cases}
z_t=d\displaystyle\int_{\R}J(x-y)z(t,y)dy-dz
+z[a_n(x)-bz],&t>T_n,\ x\in\R,\\
z(T_n,x)=K,&x\in\R.
\end{cases}
\end{equation}
 Then clearly
 \begin{equation*}
\begin{cases}
z_{nt}\geq d\displaystyle\int_{g(t)}^{h(t)}J(x-y)z_n(t,y)dy-dz_n
+z_n(a-w-bz_n),&t>T_n,\ x\in(g(t),h(t)),\\
z_n(t,g(t))\geq0,\ z_n(t,h(t))\geq0,&t>T_n,\\
z_n(T_n,x)\geq u(T_n,x),&x\in[g(T_n),h(T_n)].
\end{cases}
\end{equation*}
The comparison principle (see \cite[Lemma 2.2]{CDLL}) then infers that
\begin{equation}\label{3.25}
u(t,x)\leq z_n(t,x) \ \text{\ for\ } t\geq T_n \ \text{\ and\ } x\in[g(t),h(t)].
\end{equation}
Since $u(t,x)=0$ for $t\geq T_n$ and $x\in\R\backslash(g(t),h(t))$, we have
  \[u(t,x)\leq z_n(t,x) \ \text{\ for\ } t\geq T_n \ \text{\ and\ } x\in\R.\]
By Propositions \ref{p3.12}\,(i) and \cite[Proposition 3.4(ii)]{CDLL}, we have
\[\lim_{l\to\infty}
\lambda_p(\mathcal{L}_{(-l,l)}^d+a_n(x))
\geq\lim_{l\to\infty}
\lambda_p(\mathcal{L}_{(-l,l)}^d+\sigma_\infty)
=\sigma_\infty>0.\]
It follows from Proposition \ref{p3.12}\,(ii) that \eqref{3.24}
admits a unique positive steady state $\widetilde{z}_n\in C(\R)$:
\begin{equation}\label{3.26}
d\displaystyle\int_{\R}J(x-y)\widetilde{z}_n(y)dy-d\widetilde{z}_n
+\widetilde{z}_n(a_n(x)-b\widetilde{z}_n)=0,\ x\in\R,
\end{equation}
and
 \begin{equation}\label{3.27}
\lim_{t\to\infty}z_n(t,x)=\widetilde{z}_n(x)
\ \text{\ locally uniformly in\ } \R.
 \end{equation}
Since $\sigma_\infty/b$ is a lower solution of \eqref{3.26} and $\sigma_\infty/b\leq K$, applying the comparison principle gives that $z_n(t,x)\geq\sigma_\infty/b$ for $t\geq T_n$ and $x\in\R$. Similarly, we have $z_n(t,x)\leq K$ for $t\geq T_n$ and $x\in\R$. Thus, $\sigma_\infty/b\leq\widetilde{z}_n(x)\leq K$ for every $x\in\R$. It follows from the monotonicity of $a_n(x)$ in $n$ that $z_{n+1}(t,x)\leq z_n(t,x)$ for $t\geq T_{n+1}$ and $x\in\R$. Then $\widetilde{z}_{n+1}(x)\leq\widetilde{z}_n(x)$ for every $x\in\R$. Therefore, there exists $\widetilde{z}_\infty(x)$ such that
\begin{equation*}
\lim_{n\to\infty}\widetilde{z}_n(x)
=\widetilde{z}_\infty(x) \ \text{\ for every\ } x\in\R,
\end{equation*}
where $\widetilde{z}_\infty(x)$ satisfies $\sigma_\infty/b\leq\widetilde{z}_\infty(x)\leq K$ in $\R$.
By the Lebesgue dominant convergence theorem, we can pass to the limit in \eqref{3.26} and obtain
\begin{equation*}
d\displaystyle\int_{\R}J(x-y)\widetilde{z}_\infty(y)dy
-d\widetilde{z}_\infty+\widetilde{z}_\infty
(\sigma_\infty-b\widetilde{z}_\infty)=0,\ x\in\R.
\end{equation*}
Since this problem has a unique positive solution, we necessarily have
$\widetilde{z}_\infty(x)\equiv\sigma_\infty/b$, which implies that
\begin{equation*}
\lim_{n\to\infty}\widetilde{z}_n(x)
={\sigma_\infty}/b \ \text{\ for every\ } x\in\R.
\end{equation*}
Since $\tilde z_n$ is monotone in $n$, thanks to Dini's theorem, we have
\[\lim_{n\to\infty}\widetilde{z}_n(x)
={\sigma_\infty}/b \ \text{\ locally uniformly in\ } \R.\]
It follows from this fact, \eqref{3.25}, \eqref{3.27} and the arbitrariness of $\varepsilon$ that
\[\limsup_{t\to\infty}u(t,x)
\leq{[a-m]_+}/b
\ \text{\ locally uniformly in\ } \R.\]

(ii) For any given $l>L_1:={2}/(a-M)$. Clearly
\[\lambda_p\left(\mathcal{L}_{(-l,l)}^d
+a-\left(M+1/l\right)\right)>
\lambda_p\left(\mathcal{L}_{(-l,l)}^d
+(a-M)/2\right).\]
By \cite[Proposition 3.4]{CDLL}, there exists $L_2\geq0$ such that
 \[\lambda_p\left(\mathcal{L}_{(-l,l)}^d
+(a-M)/2\right)>0 \ \text{\ when\ } l>L_2.\]
Take $L:=\max\{L_1,L_2\}$. Then
\[\lambda_p\left(\mathcal{L}_{(-l,l)}^d
+a-\left(M+1/l\right)\right)>0 \ \text{\ when\ } l>L.\]
For any such $l$, it follows from \eqref{3.23} that there exists $T_l$ such that
 \[w(t,x)\leq M+1/l \ \ \text{\ for\ } t\geq T_l,\ x\in[-l,l].\]
Since $h(t)\to\infty$ and $g(t)\to-\infty$ as $t\to\infty$,
  there exists $\widetilde{T}_l\geq T_l$ such that
\[(-l,l)\subset(g(t),h(t))\ \ \text{\ for\ } t\geq\widetilde{T}_l.\]
Thus $u$ satisfies
 \[u_t\geq d\displaystyle\int_{-l}^lJ(x-y)u(t,y)dy-du
+u\left[a-\left(M+1/l\right)-bu\right],\ \ t>\widetilde{T}_l, \ x\in[-l,l].\]
Let $\underline{u}_l$ be the unique solution of
 \begin{equation}\label{3.28}
\begin{cases}
\underline{u}_t=d\displaystyle\int_{-l}^lJ(x-y)\underline{u}(t,y)dy
-d\underline{u}+\underline{u}\left[a-\left(
M+1/l\right)-b\underline{u}\right],
&t>\widetilde{T}_l,\ x\in[-l,l],\\
\underline{u}(\widetilde{T}_l,x)= u(\widetilde{T}_l,x),&x\in[-l,l].
\end{cases}
 \end{equation}
By the comparison principle \cite[Lemma 3.3]{CDLL},
\[u(t,x)\geq\underline{u}_l(t,x) \ \ \text{\ for\ } t\geq \widetilde{T}_l,\ x\in[-l,l].\]
By Lemma \ref{l3.13}, we have that \eqref{3.28} has a unique positive steady state $\underline{U}_l(x)$ and
  \bess
  &\lim_{t\to\infty}\underline{u}_l(t,x)=\underline{U}_l(x)
   \ \ \text{\ uniformly in\ } [-l,l],&\\
 &\lim_{l\to\infty}\underline{U}_l(x)
 =(a-M)/b \ \text{\ locally uniformly in\ } \R.&\eess
Consequently,
\begin{equation*}
\liminf_{t\to\infty}u(t,x)\geq(a-M)/b
\ \text{\ locally uniformly in\ } \R.
\end{equation*}
This completes the proof.
\end{proof}

We are now ready to prove Theorem \ref{th1.4}.

\medskip
{\it Proof of Theorem {\rm\ref{th1.4}\,(i):} The weak competition case}.

\textbf{Step 1}:
Let $q(t)$ be the solution of
\begin{equation*}
\begin{cases}
q'(t)=q(a_1-b_1q),&t>0,\\
q(0)=\sup_{x\in\R}u_0(x).&
\end{cases}
\end{equation*}
Then $\dd\lim_{t\to\infty}q(t)=a_1/b_1$.
By the comparison principle (\cite[Lemma 2.2]{CDLL}), we have $u_1(t,x)\leq q(t)$ for $t>0$ and $x\in[g(t),h(t)]$. In view of $u_1(t,x)=0$ for $t>0$ and $x\in\R\backslash(g(t),h(t))$, we have $u_1(t,x)\leq q(t)$ for $t>0$ and $x\in\R$. Hence,
 \begin{equation}\label{3.29}
\limsup_{t\to\infty}u_1(t,x)\leq{a_1}/{b_1}=:\bar A_1 \ \text{\ locally uniformly in\ } \R.
\end{equation}

\textbf{Step 2}:
By the condition $c_1/b_2<a_1/a_2<b_1/c_2$, we have
  \[a_2-c_2\bar A_1=(a_2b_1-a_1c_2)/{b_1}>0.\]
It follows from this fact, \eqref{3.29} and Lemma \ref{l3.14} that
 \begin{equation}\label{3.30}
\liminf_{t\to\infty}u_2(t,x)\geq
(a_2-c_2\bar A_1)/{b_2}=:\underline B_1
\ \text{\ locally uniformly in\ } \R.
 \end{equation}
The condition $c_1/b_2<a_1/a_2<b_1/c_2$ implies
  \[a_1-c_1\underline B_1=a_1-\frac{c_1}{b_1b_2}(a_2b_1-a_1c_2)
 =a_1-\frac{a_2c_1}{b_2}+\frac{a_1c_2c_1}{b_1b_2}>0.\]
This fact combined with \eqref{3.30} and Lemma \ref{l3.14} allows us to derive
 \bess
\limsup_{t\to\infty}u_1(t,x)\leq
(a_1-c_1\underline B_1)/{b_1}=:\bar A_2
\ \text{\ locally uniformly in\ } \R.
 \eess
Furthermore, the condition $c_1/b_2<a_1/a_2<b_1/c_2$ implies
 \[a_2-c_2\bar A_2=a_2-\frac{a_1c_2}{b_1}+\frac{c_1c_2}{b_1^2b_2}(a_2b_1-a_1c_2)>0.\]
Similar to the above,
 \[\liminf_{t\to\infty}u_2(t,x)\geq
(a_2-c_2\bar A_2)/{b_2}=:\underline B_2
\ \text{\ locally uniformly in\ } \R.\]

\textbf{Step 3}: Repeating the above procedure, we can find two sequences
$\bar A_i$ and $\underline B_i$ such that
 \[\limsup_{t\to\infty}u_1(t,x)\leq\bar A_i,\ \
\ \liminf_{t\to\infty}u_2(t,x)\geq\underline B_i
\ \text{\ locally uniformly in\ } \R,\]
and
 \[\bar A_{i+1}=(a_1-c_1\underline B_i)/{b_1},\ \
 \ \underline B_i=(a_2-c_2\bar A_i)/{b_2},\ \
 \ i=1,2,\cdots.\]
Let
 \[p:=\frac{a_1}{b_1}-\frac{a_2c_1}{b_1b_2},\; \ \ \
q:=\frac{c_1c_2}{b_1b_2}.\]
Then $p>0$, $0<q<1$ by the weak competition assumption. By direct calculation,
 \[\bar A_{i+1}=p+q\bar A_i,\;\;\; i=1,2,\cdots.\]
From $\bar A_2<\bar A_1$ and the above iteration formula, we immediately obtain
 \[0<\bar A_{i+1}<\bar A_i,\;\; \ i=1,2,\cdots,\]
from which it easily follows that
 \[\lim_{i\to\infty}\bar A_i=
\frac{a_1b_2-a_2c_1}{b_1b_2-c_1c_2},\ \ \Longrightarrow \
\lim_{i\to\infty}\underline B_i=
\frac{a_2b_1-a_1c_2}{b_1b_2-c_1c_2}.\]
Thus we have
 \[\limsup_{t\to\infty}u_1(t,x)\leq\frac{a_1b_2-a_2c_1}{b_1b_2-c_1c_2},\ \
\ \liminf_{t\to\infty}u_2(t,x)\geq\frac{a_2b_1-a_1c_2}{b_1b_2-c_1c_2}
\ \text{\ locally uniformly in\ } \R.\]

Similarly, we can show
 \[\liminf_{t\to\infty}u_1(t,x)\geq\frac{a_1b_2-a_2c_1}{b_1b_2-c_1c_2},\ \
\ \limsup_{t\to\infty}u_2(t,x)\leq\frac{a_2b_1-a_1c_2}{b_1b_2-c_1c_2}
\ \text{\ locally uniformly in\ } \R.\]
Theorem \ref{th1.4}\,(i) is proved.

\medskip
{\it Proof of Theorem {\rm\ref{th1.4}\,(ii):} The weak predation case}.

\textbf{Step 1}:
Similarly to the proof of Theorem {\rm\ref{th1.4}\,(i),
 \[\limsup_{t\to\infty}u_1(t,x)\leq{a_1}/{b_1}
 =:\bar A_1 \ \text{\ locally uniformly in\ } \R.\]
Taking advantage of Lemma \ref{l3.14} one has
 \[\limsup_{t\to\infty}u_2(t,x)\leq
(a_2+c_2\bar A_1)/{b_2}=:\bar B_1
\ \text{\ locally uniformly in\ } \R.\]
The condition $a_1b_1b_2>a_2b_1c_1+a_1c_1c_2$ implies
 \[a_1-c_1\bar B_1=\frac{a_1b_1b_2-a_2b_1c_1-a_1c_1c_2}{b_1b_2}>0.\]
Making use of Lemma \ref{l3.14}, repeatedly, we have
 \[\liminf_{t\to\infty}u_1(t,x)\geq
(a_1-c_1\bar B_1)/{b_1}=:\underline A_1
\ \text{\ locally uniformly in\ } \R,\]
and
 \[\liminf_{t\to\infty}u_2(t,x)\geq(a_2+c_2\underline A_1)/{b_2}
=:\underline B_1\ \text{\ locally uniformly in\ } \R.\]
Notice
 \[a_1-c_1\underline B_1=\frac{b_1b_2(a_1b_1b_2-a_2b_1c_1-a_1c_1c_2)+a_2b_1c_1^2c_2
 +a_1c_1^2c_2^2}{b_1^2b_2^2}>0,\]
it follows from Lemma \ref{l3.14} that
 \[\limsup_{t\to\infty}u_1(t,x)\leq
(a_1-c_1\underline B_1)/{b_1}=:\bar A_2
\ \text{\ locally uniformly in\ } \R.\]
Similar to the above,
\begin{equation*}
\limsup_{t\to\infty}u_2(t,x)\leq
(a_2+c_2\bar A_2)/{b_2}=:\bar B_2
\ \text{\ locally uniformly in\ } \R.
\end{equation*}

\textbf{Step 2}:
Repeating the above procedure, we can find four sequences
$\underline A_i,\ \bar A_i,\ \underline B_i$ and $\bar B_i$ such that, for all $i\geq 1$,
\begin{equation}\label{3.31}
\underline A_i\leq\liminf_{t\to\infty}u_1(t,x)
\leq\limsup_{t\to\infty}u_1(t,x)\leq\bar A_i, \ \
\ \underline B_i\leq\liminf_{t\to\infty}u_2(t,x)
\leq\limsup_{t\to\infty}u_2(t,x)\leq\bar B_i
\end{equation}
locally uniformly in $\R$, and
 \[\bar A_1=\frac{a_1}{b_1},\;\
 \ \bar B_i=\frac{a_2+c_2\bar A_i}{b_2},\;\
\ \underline A_i=\frac{a_1-c_1\bar B_i}{b_1},\;\
\ \underline B_i=\frac{a_2+c_2\underline A_i}{b_2},\;\
\ \bar A_{i+1}=\frac{a_1-c_1\underline B_i}{b_1},\]
$i=1,2,\cdots$. Define
 \[a:=\frac{a_2b_1+a_1c_2}{b_1b_2},\quad\ q:=\frac{c_1c_2}{b_1b_2}.\]
By direct calculation, we have
 \bess
 &\bar B_1=a,\ \ \ \ \dd\underline A_1=\frac{a_1}{b_1}-\frac{c_1}{b_1}a,\ \ \ \
\underline B_1=a(1-q),&\\[2mm]
 &\bar A_2=\dd\frac{a_1}{b_1}-\frac{c_1}{b_1}a(1-q),\ \ \
\bar B_2=a(1-q+q^{2}),\ \ \ \underline A_2=\frac{a_1}{b_1}-\frac{c_1}{b_1}
a(1-q+q^{2}),&
 \eess
and
 \[\underline B_{i+1}=a(1-q)+q^2\underline B_i,\ \ \ \ \
\bar B_{i+1}=a(1-q)+q^2\bar B_i,\; \;\; i\geq1.\]
Since $a_1b_1b_2>a_2b_1c_1+a_1c_1c_2$, we have $0<q<1$ and
 \[\underline B_2>\underline B_1>0,\; \ \ \ \bar B_1>\bar B_2>0.\]
The above iteration formula then infers
 \[\underline B_{i+1}>\underline B_i>0,\; \ \ \ \bar B_i>\bar B_{i+1}>0, \ \ \ i\geq 1.\]
 From these we easily obtain
 \begin{equation}\label{3.32}
\lim_{i\to\infty}\bar B_i=
\lim_{i\to\infty}\underline B_i=
\frac{a_1c_2+a_2b_1}{b_1b_2+c_1c_2},
\end{equation}
and subsequently
 \begin{equation}\label{3.33}
\lim_{i\to\infty}\bar A_i=\lim_{i\to\infty}\underline A_i=\frac{a_1b_2-a_2c_1}{b_1b_2+c_1c_2}.
 \end{equation}
Theorem \ref{th1.4}\,(ii) clearly is a consequence of \eqref{3.31}, \eqref{3.32} and \eqref{3.33}. The proof is finished.


\begin{thebibliography}{99}
\bibliographystyle{siam}
\setlength{\baselineskip}{15pt}

\bibitem{BL}
X. Bai and F.  Li, {\it
Classification of global dynamics of competition models with nonlocal dispersals I: symmetric kernels},
Calc. Var. Partial Differential Equations  {\bf 57} (2018), no. 6, Art. 144, 35 pp.

\bibitem{BLS}
X. Bao, W.-T. Li and W. Shen, {\it Traveling wave solutions of Lotka-Volterra competition systems with nonlocal dispersal in periodic habitats}, J. Differential Equations {\bf 260} (2016), no. 12, 8590-8637.

\bibitem{BZ07}P. Bates, G. Zhao, {\it Existence, uniqueness and stability of the stationary solution to a nonlocal evolution equation arising in population dispersal}, J. Math. Anal. Appl., \textbf{332} (2007), 428-440.

\bibitem{B-JFA16}H. Berestycki, J. Coville, H. Vo, {\it On the definition and the properties of the principal eigenvalue of some nonlocal operators}, J. Funct. Anal., \textbf{271}
(2016), 2701-2751.

\bibitem{B-JMB16}H. Berestycki, J. Coville, H. Vo,
{\it Persistence criteria for populations with non-local dispersion},
{J. Math. Biol. }\textbf{72} (2016), 1693-1745.


\bibitem{CDLL} J.-F. Cao, Y. H. Du, F. Li, W.-T. Li, {\it The dynamics of a Fisher-KPP nonlocal
diffusion model with free boundaries}, J. Functional Anal., in press (https:/\!/doi.org/10.1016/j.jfa.2019.02.013).

\bibitem{Cov}J. Coville, {\it On a simple criterion for the existence of a principal eigenfunction of some nonlocal operators}, J. Differential Equations, \textbf{249} (2010), 2921-2953.

\bibitem{DLin10}Y. H. Du and Z. G. Lin, {\it Spreading-Vanishing dichotomy in the diffusive logistic model with a free boundary}, SIAM J. Math. Anal., \textbf{42} (2010), 377-405.



\bibitem{GW} J. Guo, C. H. Wu, {\it On a free boundary problem for a two-species weak competition system}, J. Dynam. Differential Equations \textbf{24}
(2012), 873-895.

\bibitem{HNS} G. Hetzer, T. Nguyen and W. Shen, {\it Coexistence and extinction in the Volterra-Lotka competition model with nonlocal dispersal}, Commun. Pure Appl. Anal. {\bf 11} (2012), 1699-1722.

\bibitem{H-JMB03}V. Hutson, S. Martinez, K.Mischaikow, G. Vickers, {\it The evolution of dispersal}, J. Math. Biol., \textbf{47} (2003), 483-517.

\bibitem{KLS10}C. Kao, Y. Lou, W. Shen, {\it Random dispersal vs. non-local dispersal}, Discrete Contin. Dyn. Syst., \textbf{26} (2010), 551-596.

\bibitem{Nathan12}R. Natan, E. Klein, J. J. Robledo-Arnuncio, E. Revilla, Dispersal kernels: Review, in {\it Dispersal Ecology and Evolution}, J. Clobert, M. Baguette, T. G. Benton and J. M. Bullock, eds., Oxford University Press, Oxford, UK, 2012, pp. 187-210.

\bibitem{Wjde14}M. X. Wang, {\it On some free boundary problems of the
prey-predator model}, J. Differential Equations, \textbf{256} (2014), 3365-3394.




\bibitem{WZhang16} M. X. Wang and Y. Zhang,
{\it The time-periodic diffusive competition models with a free boundary and sign-changing growth rates}, Z. Angew. Math. Phys. {\bf 67} (2016), no. 5, Art. 132, 24 pp.

\bibitem{WZjdde14}
M. X. Wang and J. F.  Zhao, {\it Free boundary problems for a Lotka-Volterra competition
system}, {J. Dyn. Diff. Equat.}, \textbf{26}(3)(2014), 655-672.

\bibitem{WZjdde17} M. X. Wang and J. F. Zhao, {\it A free boundary problem for a predator-prey model with double free boundaries}, J. Dyn. Diff. Equat., \textbf{29} (3)(2017), 957-979.

\bibitem{ZW2015} Y. Zhang and M. X. Wang, {\it A free boundary problem of the ratio-dependent prey-predator model}, Appl. Anal. \textbf{94} (10)(2015), 2147-2167.

\bibitem{ZhaoW16} Y. G. Zhao and M. X. Wang, {\it Free boundary problems for the diffusive competition system in higher dimension with sign-changing coefficients}, IMA J. Appl. Math., {\bf 81} (2016), 255-280.
\end{thebibliography}
\end{document}